\documentclass[11pt]{amsart}
\usepackage{amsmath,amssymb,enumerate}

\newtheorem{prop}{Proposition}[section]
\newtheorem{coro}[prop]{Corollary}

\newtheorem{lem}[prop]{Lemma}
\newtheorem{rem}[prop]{Remark}

\newtheorem{defi}[prop]{Definition}
\newtheorem{theo}[prop]{Theorem}

\newcommand{\R}{\mathbb R}
\newcommand{\Q}{\mathbb Q}
\newcommand{\C}{\mathbb C}
\newcommand{\N}{\mathbb N}

\renewcommand{\a}{\alpha}
\renewcommand{\b}{\beta}

\newcommand{\de}{\delta}
\newcommand{\e}{\varepsilon}
\newcommand{\f}{\varphi}

\newcommand{\p}{\psi}

\begin{document}

\title[Viscosity solutions to degenerate CMAE]{Viscosity solutions to degenerate complex Monge-Amp\`ere equations}

\author{Philippe Eyssidieux, Vincent Guedj, Ahmed Zeriahi }
\date{\today}

\address{Institut Fourier, Universit\'e Grenoble I, France}

\email{Philippe.Eyssidieux@ujf-grenoble.fr }

\address{LATP, Universit{\'e} Aix-Marseille I\\
13453 Marseille Cedex 13\\
 France}

\email{guedj@cmi.univ-mrs.fr}

\address{Institut de Math\'ematiques de Toulouse,  U.P.S.\\
F-31062 Toulouse cedex 09\\}

\email{zeriahi@math.ups-tlse.fr}

\begin{abstract} 
Degenerate complex Monge-Amp\`ere equations on compact K\"ahler manifolds have been 
recently intensively studied using tools from pluripotential theory. 
We develop an alternative approach based on the concept of viscosity solutions
 and
  compare systematically viscosity concepts with pluripotential theoretic ones. 

This approach works only for a rather restricted type of degenerate complex Monge-Amp\`ere 
equations. Nevertheless, we prove that the local potentials of the singular K\"ahler-Einstein 
metrics
constructed previously by the authors are continuous plurisubharmonic functions. 
They were previously known to be locally bounded. 

Another application is a  lower order construction with a $C^0$-estimate of  the solution 
to the Calabi conjecture which  does not use  
Yau's celebrated theorem. 
\end{abstract}

\maketitle

\section*{Introduction}

Pluripotential theory lies at the fundation of the  approach to degenerate complex 
Monge-Amp\`ere equations on compact
 K\"ahler manifolds as developed in \cite{Kol}, \cite{EGZ1,EGZ2}, 
\cite{TZ}, \cite{Z}, \cite{DP}, \cite{BEGZ}, \cite{BBGZ} and many others. 
This method is global in nature, since it relies on 
some delicate integrations by parts. 

On the other hand, a standard PDE approach to  second-order degenerate elliptic equations
is the method of viscosity solutions introduced in \cite{L},  see \cite{CIL} for a survey. 
This method is local in nature - and 
solves existence and unicity problems for weak solutions very efficiently. 
Our main goal in this article is to develop the viscosity approach for complex Monge-Amp\`ere equations
on compact complex manifolds.

Whereas the viscosity theory for real Monge-Amp\`ere equations has been developed by
 P.L. Lions and others (see e.g.\cite{IL}), 
the complex case hasn't been studied   until very recently. There is a viscosity approach to the Dirichlet problem for the complex Monge-Amp\`ere equation on a smooth hyperconvex domain in  a Stein manifold  in \cite{HL}. This recent article does not however prove any new results for complex Monge-Amp\`ere equations since this case serves there as a motivation to develop a deep
generalization of plurisubharmonic functions to Riemannian 
manifolds with some special geometric structure (e.g. exceptional holonomy group). To the best of our knowledge, there is no reference on viscosity solutions  to complex Monge-Amp\`ere equations on compact K\"ahler manifolds. 

There has been some recent interest in adapting viscosity methods to solve degenerate elliptic equations on compact or complete Riemannian manifolds \cite{AFS}. This  theory can be applied to complex Monge-Amp\`ere equations only in very restricted cases since 
it requires the Riemann curvature tensor to be nonnegative. Using \cite{Mok}, a compact K\"ahler manifold with a non-negative 
Riemannian curvature tensor has an \'etale cover
which is a product of a symmetric space of compact type (e.g.:  
 $\mathbb{P}^n(\C)$, grassmannians) and a  compact complex torus.
In particular, \cite{AFS} does not allow in general to construct a viscosity solution to the elliptic equation:
$$ 
(\omega+dd^c\f)^n=e^{\f}v         \leqno{(DMA)_{\omega,v}}
$$
where $\omega$ is a smooth K\"ahler form and $v$ a smooth volume on a general $n$-dimensional compact K\"ahler manifold $X$. A unique smooth solution has been however known to exist for more than thirty years thanks to the celebrated work of Aubin and Yau, \cite{Aub} \cite{Yau}.
This is a strong indication that the viscosity method should  work in this case to produce easily weak solutions
%\footnote{Higher regularity theory is technically more involved and we do not improve it.}.

In this article, we confirm this guess, define and study viscosity solutions to degenerate complex Monge-Amp\`ere equations.  Our main technical result is: 

\vskip.2cm

\noindent{\bf{Theorem A.}} 
{\em  Let $X$ be a compact complex manifold,
 $\omega$ a continuous closed real (1,1)-form with 
${\mathcal C}^2$
local potentials and $v>0$ be a volume form with continuous density.  
Then the viscosity comparison principle holds for $(DMA)_{\omega,v}$. }

\vskip.2cm

The viscosity comparison principle (see below for details) 
differs substantially from the pluripotential 
comparison principle of \cite{BT2}  which is the main tool in
\cite{Kol}, \cite{GZ1}, \cite{EGZ1}. 
 This technical statement  is based on the Alexandroff-Bakelmann-Pucci maximum principle. We need however 
to modify the argument in \cite{CIL} by a localization technique.

\smallskip

Although we need to assume $v$ is positive in Theorem A, it is then easy to let it degenerate to a non negative density
in the process of constructing weak solutions to degenerate complex Monge-Amp\`ere equations. We obtain this way
the following:

\vskip.2cm

\noindent{\bf{Corollary B.}} 
{\em
Assume $X$ is as above, $v$ is merely semi-positive with $\int_X v >0$.
If $\omega\ge 0$ and $\int_X \omega^n>0 $ , then
there is a unique  viscosity solution  $\f \in C^0(X)$
to $(DMA)_{\omega,v}$. 

If $X$ is a compact complex manifold in the Fujiki class, it co\"incides with the unique
 locally bounded $\omega$-psh function $\f$ on $X$ such that 
$(\omega+dd^c\f)^n_{BT}=e^{\f} v$ in the pluripotential sense  \cite{EGZ1}.
}

\vskip.2cm

Recall that $\f$ is $\omega$-plurisubharmonic ($\omega$-psh for short) if it is an u.s.c. integrable function
such that $\omega+dd^c \f \geq 0$ in the weak sense of currents.

It was shown in this context by Bedford and Taylor  \cite{BT2} that
when $\f$ is bounded,
there exists a unique positive Radon measure $(\omega+dd^c\f)^n_{BT}$ with the following
property:  if $\f_j$ are smooth, locally $\omega$-psh and decreasing to $\f$, then the smooth measures
$(\omega+dd^c \f_j)^n$ weakly converge towards the measure $(\omega+dd^c\f)^n_{BT}$.
If the measures $(\omega+dd^c\f_j)^n$ (locally) converge to $e^{\f} v$, we say that 
the equality $(\omega+dd^c\f)^n_{BT}=e^{\e} v$ holds  {\it in the
pluripotential sense}.

Combining pluripotential and viscosity techniques, we can push our results further
and obtain the following:

\vskip.2cm

\noindent{\bf{Theorem C.}} 
{\em
 Let $X$ be a compact complex manifold in the Fujiki class. Let $v$ is be a semi-positive 
probability measure with $L^p$-density, $p>1$, and fix
$\omega\ge 0$ a smooth closed real semipositive $(1,1)$-form
such that $\int_X \omega^n=1 $.
The unique locally bounded $\omega$-psh function on $X$ normalized by
 $\int_X\f =0$ such 
that its Monge-Amp\`ere measure satisfies
 $(\omega+dd^c\f)^n_{BT}= v$ is continuous. 
}

\vskip.2cm

This continuity statement was obtained in \cite{EGZ1} under a regularization statement
for  $\omega$-psh functions that we were not able to obtain in full generality.
It could have been obtained using \cite{ AFS} in the  cases covered by this reference. 
However, for rational homogeneous spaces,  the regularization statement is  easily proved by convolution \cite{Dem2} and \cite{AFS} does not give anything new. 
A proof of the continuity when $X$ is projective  under mild technical assumptions
has been obtained in  \cite{DZ}.

\smallskip

We now describe the organization of the article.
The first section is devoted to the local theory. It makes the connection between 
the complex Monge-Amp\`ere operator and the 
viscosity subsolutions of inhomogenous complex Monge-Amp\`ere equations. 
We have found no reference for these basic facts.
 
In the second section, we introduce the viscosity comparison principle
 and give a proof of the main theorem . The gain 
with respect to classical pluripotential theory is that one can consider
 supersolutions  to prove continuity of pluripotential solutions for  $(\omega+dd^c \f)^n= e^{\f} v$. 

In the third section we apply these ideas to show that the singular K\"ahler-Einstein potentials 
constructed in \cite{EGZ1} are globally continuous. 

In the fourth and last section, we  stress some advantages of our method:
\begin{itemize}
\item it gives an alternative proof of Kolodziej's $C^0$-Yau theorem which does not depend on \cite{Yau}.  
\item it allows us to easily produce the unique negatively curved singular K\"ahler-Einstein metric
in the canonical class of a projective manifold of general type, a result obtained first in \cite{EGZ1} assuming \cite{BCHM}, then in \cite{BEGZ} by means
of asymptotic Zariski decompositions.
\end{itemize}
Then, we
establish further comparison principles: these could be useful when studying
similar problems were pluripotential tools do not apply.
 We end the article by some remarks on a possible interpretation of
 viscosity supersolutions in terms of pluripotential
theory  using Berman's plurisubharmonic
projection.

\medskip

The idea of applying viscosity methods to the  K\"ahler-Ricci flow was
 proposed originally in a remark from the preprint \cite{CLN}.
We hope that the technique developed here will have further applications. 
In a forthcoming work it will be applied to the K\"ahler-Ricci flow.

\section{Viscosity subsolutions
 to  $ (dd^c\f)^n= e^{\e \f} v$. } \label{sec:local}

The purpose of this section is to make the connection, in a purely
local situation,
 between the pluripotential theory of complex Monge-Amp\`ere
operators, as founded by Bedford-Taylor
 \cite{BT2}, and the concept of viscosity subsolutions
 developed by Lions et al.
(see \cite{IL,CIL}). 

\subsection{Viscosity subsolutions of $(dd^c\f)^n=v$}

Let $M = M^{(n)}$ be a (connected) complex 
manifold of dimension $n$ and $v$ a semipositive 
measure with continuous density.
In this section $B$ will denote the unit ball of $\C^n$ or 
its image under a coordinate chart in $M$.

\begin{defi} An upper semicontinuous function $\f: M\to \R \cup \{-\infty\}$ 
is said to be a viscosity subsolution of the Monge-Amp\`ere equation 
$$
(dd^c\f)^n = v \leqno (DMA)_v
$$
if it satisfies the following conditions
\begin{enumerate}
\item  $\f|_{M}\not \equiv -\infty$.
\item For every $x_0\in M$ and any ${\mathcal C}^2$-function $q$ defined on a neighborhood of $x_0$ such that 
$ \f-q \  \text{ has a local maximum at} \ x_0 $ then 
$$
(dd^c q)^n_{x_0} \ge v_{x_0}.
$$ 
\end{enumerate}
We will also say that $\f$ satisfies the differential inequality $(dd^c\f)^n \ge v$ in the viscosity sense on $M$.
\end{defi}

Note that if $v\ge v'$ then $(dd^c\f)^n \ge v$ in the viscosity sense implies  $(dd^c\f)^n \ge v'$. This holds
in particular  if $v'=0$. 

Another basic observation is that the class of subsolutions is stable under taking maximum:

\begin{lem} \label{lem:maxsubsolution}
If $\f_1,\f_2$ are subsolutions of $(dd^c\f)^n = v$, so is $\sup(\f_1,\f_2)$.
\end{lem}

The proof is straightforward and left to the reader. We now observe that a function
$\f$ satisfies $(dd^c \f)^n \geq 0$ in the viscosity sense if and only if it is
plurisubharmonic:

\begin{prop} \label{viscpsh}The viscosity
 subsolutions of $(dd^c\f)^n=0$ are precisely the plurisubharmonic functions on $M$. 
\end{prop}

\begin{proof}
The statement is local and we can assume $M=B$.  
Let $\f$ be a subsolution of $(dd^c\f)^n=0$.  Let $x_0\in B$ such that $\f(x_0)\not= -\infty$. 
Let $q\in {\mathcal C}^2(B)$ such that $\f-q$ has a local maximum at $x_0$. Then the hermitian matrix $Q=dd^cq_{x_0}$
satisfies      $\det(Q)\ge 0$. For every hermitian semipositive matrix $H$, we also have $\det(Q+H)\ge 0$ since,
a fortiori for  
$q_H=q+H(x-x_0)$, $\f-q_H$ has a local maximum at $x_0$ too.

It follows from Lemma \ref{lem:semipositive} below that $Q=dd^cq_{x_0}$ is actually semipositive.
We infer that for every positive definite hermitian matrix $(h^{i\bar j})$
$\Delta_H q (x_0) := h^{i\bar j} \frac{\partial^2 q}{\partial z_i \partial \bar z_j} (x_0) \ge 0$, i.e.: $\f$ is a viscosity subsolution 
of $ \Delta_H \f =0$. In appropriate complex coordinates 
this constant coefficient differential operator is nothing but the Laplace operator.  Hence, \cite{Hor} prop 3.2.10' p. 147 applies to the effect that
$\f$ is $\Delta_H$-subharmonic hence is in $L^1_{loc}(B)$ and satisfies 
$\Delta_H  \f \ge 0$ in the sense of distributions. Let $(w^i)$ be any vector in $\C^n$. Consider a 
positive hermitian matrix $(h^{i\bar j})$ degenerating to the rank one matrix $(w^i\bar w^j)$. By continuity, we have $w^i\bar w^j \frac{\partial^2 \f}{\partial z_i \partial \bar z_j}\ge 0$ in the sense of distributions.  Thus $\f$ is plurisubharmonic. 

\smallskip

Conversely, assume $\f$ is plurisubharmonic. Fix $x_0\in B$, $q\in {\mathcal C}^2(B))$ such that
 $
  \f-q \  \text{ has a local maximum at} \ x_0 .
  $
 Then, for every small enough ball $B'\subset B$ centered at $x_0$, we have 
 $$
 \f(x_0)-q(x_0) \ge \frac{1}{V(B')} \int_{B'} (\f-q) \, dV,
 $$ 
 hence
 $$
 \frac{1}{V(B')} \int_{B'} q \, dV -q(x_0) \ge \frac{1}{V(B')} \int_{B'} \f  \, dV-\f(x_0) \ge 0. 
 $$
 Letting the radius of $B'$ tend to $0$, it follows since $q$ is ${\mathcal C}^2$ that $\Delta q_{x_0} \ge 0$. Using complex ellipso\"ids instead of balls\footnote{This amounts to a linear change of complex coordinates.},  we conclude that $\Delta_H q (x_0)  \ge 0$ for every positive definite hermitian matrix.
Thus $dd^c q _{x_0} \ge 0$ and  $(dd^c\f)^n \ge 0$ in the viscosity sense. 
\end{proof}

The following lemma is easily proven by diagonalizing $Q$:

\begin{lem} \label{lem:semipositive}
Let $Q$ be an hermitian matrix such that, for every semipositive hermitian matrix $H$, $\det(Q+H)\ge 0$ then $Q$ is semipositive.
\end{lem}

Recall that when $\f$ is plurisubharmonic and locally bounded, its Monge-Amp\`ere measure
$(dd^c \f)_{BT}^n$ is well defined \cite{BT2} (as the unique
 limit of the smooth measures $(dd^c\f_j)^n$,
where $\f_j$ is any sequence of smooth psh functions decreasing to $\f$).
Our next result makes the basic connection between this pluripotential notion and its viscosity counterpart:

\begin{prop} \label{pro:visc=pluripot}
Let $\f$ be a locally bounded upper semi-continuous function in $M$.
It satisfies $(dd^c\f)^n \ge v$ in the viscosity sense iff 
it is plurisubharmonic and its Monge-Amp\`ere measure satisfies 
$(dd^c\f)^n_{BT} \ge v$ in the pluripotential sense.
\end{prop}

\begin{proof}
We first recall the following classical formulation of the pluripotential comparison principle for the complex Monge-Amp\`ere operator,
acting on bounded plurisubharmonic functions \cite{BT2}:

\begin{lem}\label{cpd}
Let $u,w\in PSH\cap L^{\infty} (B) $. 

If $u\ge w$ near $\partial B$ and $(dd^c u)^n_{BT}\le (dd^c w)^n_{BT}$ then $u\ge w$. 
\end{lem}

Assume $\f \in PSH\cap L^{\infty} (B)$ satisfies $(dd^c \f)^n_{BT}\ge v$. Consider $q$ a ${\mathcal C}^2$ function such that
$\f-q$ achieves a local maximum at $x_0$ and $\f(x_0)=q(x_0)$. Since $\f$ satisfies $(dd^c\f)^n\ge 0$ in the 
viscosity sense,  $(dd^cq)_{x_0}^n \ge 0$ and 
$dd^cq_{x_0} \ge 0$ by lemma 	\ref{lem:semipositive}. Assume $(dd^cq)_{x_0}^n < v_{x_0}$. 
Let $q^{\e}:=q+ \e\| x-x_0 \|^2 $. Choosing $\e>0$ small enough, we have  $0<(dd^cq^{\e}_{x_0})^n < v_{x_0}$. 
Since $v$ has continuous density, we can chose a small ball $B'$ containing $x_0$ of radius $r>0$ such that
$\bar q^{\e}=q^{\e}-\e \frac{r^2}{2} \ge \f$ near $\partial B'$ and 
$(dd^c \bar q^{\e})^n_{BT} \le (dd^c \f)^n_{BT}$. 
The comparison principle (Lemma \ref{cpd})
yields $\bar q^{\e} \ge \f$ on $B'$. But this fails at $x_0$. Hence $(dd^c q )^n_{x_0} \ge v_{x_0}$ and $\f$ 
is a viscosity subsolution. 

\medskip

Conversely assume $\f$ is a viscosity subsolution.  Fix $x_0\in B$ such that $\f(x_0)\not= -\infty$
and $q\in {\mathcal C}^2$ such that $\f-q$ has a local maximum at $x_0$. Then the hermitian matrix $Q=dd^cq_{x_0}$
satisfies      $\det(Q)\ge v_{x_0}$.

Recall that the classical trick (due to Krylov) of considering the complex Monge-Amp\`ere equation as a Bellmann equation 
relies on the following:

\begin{lem} \cite{Gav}
Let $Q$ be a $n\times n $ non negative hermitian matrix, then
$$
 \det(Q)^{1/n}= \inf \{ \rm{tr}(H Q)  \, | \, H \in H_n^+ \text{ and } \det(H)=n^{-n} \},
$$ 
where $H_n^+$ denotes the set of positive hermitian $n \times n$ matrices.
\end{lem}

Applying this to our situation, it follows that for every positive definite hermitian matrix 
$(h_{i\bar j})$ with $\det(h)=n^{-n}$, $\Delta_H q (x_0) := h_{i\bar j} \frac{\partial^2 q}{\partial z_i \partial \bar z_j} (x_0) \ge v^{1/n}(x_0)$, i.e.
$\f$ is a viscosity subsolution 
of the linear equation $\Delta_H \f \ge v^{1/n}$.
 
 This is a constant coefficient linear partial differential
equation.  Assume $v^{1/n}$ is $C^{\alpha}$ with $\alpha>0$ and choose a ${\mathcal C}^2$ solution of 
$\Delta_H \f = v^{1/n}$ in a neighborhood of $x_0$. Then $u=\f-f$ 
satisfies $\Delta_H u\ge 0$ in the viscosity sense. 
Once again, \cite{Hor} prop 3.2.10' p. 147 applies to the effect that
$u$ is $\Delta_H$-subharmonic hence  
$\Delta_H \f \ge v^{1/n}$ in the sense of positive Radon measures. 

Using convolution to regularize $\f$ and setting $\f_{\e}=\f * \rho_{\e}$ we see that 
 $\Delta_H \f_{\e}\ge (v^{1/n})_{\e}$. 
 Another application of the above lemma yields 
 $$
 (dd^c\f_{\e})^n \ge ( (v^{1/n})_{\e})^n.
 $$ 
 Here $\tilde \f_k=\f_{1/k}$ is a decreasing sequence of smooth functions converging to $\f$. 
 Continuity of $(dd^c\f)^n_{BT}$ with respect to such a sequence \cite{BT2} yields $(dd^c\f)^n_{BT} \ge v$. 
 
 This settles the case when $v>0$ and $v$ is H\"older continuous. In case $v>0$ is merely continuous
 we observe that $v=\sup\{w | w \in {\mathcal C}^{\infty}, \ v\ge w>0 \}$. 
  Taking into account the fact that
 any subsolution of $(dd^c\f)^n=v$ is a subsolution of $(dd^c\f)^n =w$ provided $v\ge w$ we conclude $(dd^c\f)^n_{BT} \ge v$. 
 
 In the general case, we observe that $\psi_{\e}(z)=\f(z)+\e \| z \|^2$ satisfies $(dd^c\psi_{\e})^n \ge v+ \e^n \lambda$ in the viscosity sense with $\lambda$ the euclidean volume form. 
Hence $$(dd^c\psi_{\e})^n_{BT}\ge v.$$
 From which we conclude that $(dd^c\f)^n_{BT} \ge v$. 
 \end{proof}
 
 \begin{rem}
 The basic idea of the proof is closely related to the 
method in \cite{BT1} and is the topic treated in \cite{Wi}. 
The next section contains a 
more powerful version of 
this argument. 
However, it uses  sup-convolution which is not a conventional
tool in pluripotential theory and we felt that
keeping this version would  improve the exposition.
 \end{rem}
 
 We now relax the assumption that $\f$ being bounded and connect viscosity subsolutions to
 pluripotential subsolutions through the following:
 
 \begin{theo} \label{thm:visc=pluripot}
 Assume $v = (dd^c\rho)^n_{BT}$ for some bounded plurisubharmonic  function $\rho$.
 Let $\f$ be an upper semicontinuous function such that $\f \not \equiv -\infty$ on any connected component. The
 following are equivalent:
 
 i) $\f$ satisfies $(dd^c\f)^n\ge v$ in the viscosity sense;
 
 ii) $\f$ is plurisubharmonic and for all $c>0$,
 $(dd^c\sup[\f, \rho-c])^n_{BT}\ge v$. 
 \end{theo}
 
 Observe that these properties are local and that the semi-positive measure $v$ can always be written locally
 as $v=(dd^c\rho)^n_{BT}$ for some bounded plurisubharmonic function $\rho$ \cite{Kol}.

 \begin{proof}   
 Assume first that  $\f$ is a viscosity subsolution of $(dd^c \rho)^n= v$.
 Since $\rho-c$ is also a subsolution, it follows from Lemma \ref{lem:maxsubsolution} that $\sup(\f, \rho-c)$ is a subsolution,
 hence Proposition \ref{pro:visc=pluripot} yields $(dd^c\sup(\f, \rho-c))^n_{BT} \ge v$.    

Conversely, fix $x_0\in M$ and assume i) holds. If $\f$ is locally bounded near $x_0$, Proposition \ref{pro:visc=pluripot}
implies that $\f$ is a viscosity subsolution near $x_0$. 

Assume $\f(x_0)\not= -\infty $ but $\f$ is not locally bounded near $x_0$. Fix 
$q\in {\mathcal C}^2$ such that
$q\ge \f$ near $x_0$ and $q(x_0)=\f(x_0)$. Then for $c>0$ big enough we have $q\ge \f_c=\sup(\f, \rho-c)$ and 
$q(x_0)=\f_c(x_0)$, hence $(dd^cq) _{x_0}^n \ge v_{x_0}$ by Proposition \ref{pro:visc=pluripot} again.

Finally if $\f(x_0)=-\infty$ there are no $q$ to be tested against the differential inequality, hence it holds for every test function $q$.
  \end{proof}

 Condition ii) might seem a bit cumbersome. The point is that
 the Monge-Amp\`ere operator can not be defined on the
 whole space of plurisubharmonic functions. 
 The above arguments actually work in any class of plurisubharmonic functions in which the Monge-Amp\`ere
 operator is continuous by decreasing limits of locally bounded functions and the
 comparison principle holds. These are  precisely the finite energy classes 
 studied in
 \cite{Ceg2,GZ2}. 
 
When $\f$ belongs to its
 domain of definition, condition ii) is equivalent to $(dd^c\f)^n_{BT} \geq v$ in the pluripotential sense.
To be more precise, we have:

\begin{coro}
 Let $\Omega\subset \C^n$ be a hyperconvex domain. Then $\f\in \mathcal{E}(\Omega)$, see \cite{Ceg3} for the notation,
satisfies $(dd^c\f)^n\ge v$ in the viscosity sense iff 
its Monge-Amp\`ere measure $(dd^c \f)^n_{BT}$ satisfies  $(dd^c\f)^n_{BT}\ge v$. 
\end{coro}

We do not want to recall the definition of the class ${\mathcal E}(\Omega)$. Suffices to say that when
$n=2$, a psh function $\f$ belongs to this class if and only if $\nabla \f \in L^2_{loc}$ \cite{Blo}.

\subsection{Viscosity subsolutions to $(dd^c \f)^n=e^{\e \f}v$}

Let $\e>0$ be a real number. Say that an u.s.c. function
$\f$ is a viscosity subsolution of $(dd^c \f)^n=e^{\e \f} v$
if $\f$ is not identically $-\infty$ and for all $x_0 \in M$,
for all  $q \in {\mathcal C}^2 (M)$ 
of $x_0$ such that $\f-q$ has a local maximum 
at $x_0$ and $\f(x_0)=q(x_0)$, one has
$(dd^c q(x_0))^n \geq e^{\e q(x_0)} v(x_0)$ .

\begin{prop} \label{pro:visc=pluripot+tw}
Let $\f: M \rightarrow \R$ be a bounded u.s.c. function.
It satisfies $(dd^c \f)^n \geq e^{\e \f} v$ in the viscosity sense
if and only if it (is plurisubharmonic and it)  does in the pluripotential sense.
\end{prop}

\begin{proof}
When $\f$ is continuous, so is the density of $\tilde{v}=e^{\e \f} v$
and Proposition \ref{pro:visc=pluripot} above can be applied. When $\f$ is not assumed to be continuous,
the issue is more subtle.

We can assume without loss of generality that $\e = 1$ and $M = \Omega$ is a domain in $\C^n$. 
Assume $\f$ is a viscosity subsolution. It follows from Proposition \ref{viscpsh} that $\f$ is psh.
Set $v = f \beta_n,$ where $f > 0$ is the continuous density of the volume form $v$ w.r.t. the euclidean volume form on $\C^n$.
We approximate $\f$ by its sup-convolution:
$$
\f^{\de}(x):=\sup_y \left\{ \f(y)-\frac{1}{2\de^2} |x-y|^2 \right\}, \ \ x \in \Omega_{\de},
$$
for $\de > 0$ small enough, where $\Omega_{\de} := \{ x \in \Omega ; dist (x,\partial \Omega) > A \de\}$ and $A > 0$ is a large constant so that $A^2 > 2 \text{osc}_{\Omega} \f$.
 
This family  of
 semi-convex functions decreases towards $\f$
as $\de$ decreases to zero. Furthermore, by \cite{Ish}, 
 $\f^{\de}$ 
satisfies the following inequality in the sense of viscosity on $\Omega_{\de}$
$$
(dd^c \f^{\de})^n \geq e^{\f^{\de}} f_{\de} \beta_n,
\text{ with } f_{\de}(x)=\inf \{f(y) \, / |y-x| \leq A \de \}.
$$
It follows from Proposition \ref{viscpsh} that $\f^{\de}$ is psh 
\footnote{This argument implies that a sup convolution of a psh function is psh. 
This in turn is easily deduced from the change of variables $y=x-y'$
in the definition of $\f^{\de}(x)$. }.
Since $\f^{\de}$ is {\it continuous}, we can invoke Proposition \ref{pro:visc=pluripot}
and get that
$$
(dd^c \f^{\de})^n \geq e^{\f^{\de}} f_{\de} \, \beta_n
\geq e^{\f} f_{\de} \, \beta_n
$$ 
holds in the pluripotential sense.
Since the complex Monge-Amp\`ere operator is continuous along decreasing sequences of bounded
psh functions, and since $f_{\de}$ increases towards $f$, we finally
obtain $(dd^c \f)^n \geq e^{\f} v$ in the pluripotential sense.

We now treat the other implication. 
Let $\f$ be a psh function satisfying the inequality
$$
(dd^c \f)^n \geq e^{\f} v,
$$
in the pluripotential sense on $\Omega$.
We want to prove that $\f$ satisfies the above differential inequality in the sense of viscosity on $\Omega$. If $\f$ were continuous then we could use Proposition 1.5. 
But since $\f$ is not necessarily continuous we first approximate $f$ using regularization by convolution $\f_{\de} := \f \star \chi_{\de}$ on $\Omega_{\de}$.
Lemma \ref{convol} below yields the following pointwise inequality in $\Omega_{\de}$:
\begin{equation}
(dd^c \f_{\de})^n \geq e^{\f_{\de}}  
f_{\de} \beta_n, \ \ \text{ with } 
f_{\de} (x) := \inf \{ f (y) ; \vert y - x\vert \leq \de \}.
\label{eq:conv-ineq}
\end{equation}

 Let $x_0 \in \Omega$, $q$ be a quadratic polynomial such that $\f (x_0) = q (x_0)$ and $\f \leq q$ on a neighbourhood of $x_0$, say on a ball $2 B$, where $B := B (x_0,r) \Subset \Omega$.
Since $\f$ is psh on $\Omega$, it satisfies $(dd^c \f)^n \geq 0$ in the viscosity sense on $\Omega$ by Proposition 1.5,
hence Lemma 1.4 yields $dd^c q (x_0) \geq 0$.  Replacing $q$ by $q (x) + \e \vert x- x_0\vert^2$ and taking $r > 0$ small enough, we can assume that $q$ is psh on the ball $2 B$.
We want to prove that $(dd^c q (x_0))^n \geq e^{\f (x_0)} f (x_0) \beta_n$.
Fix $\e > 0$ and set
 $$
  q^{\e} (x) := q (x) +  2 \e(\vert x - x_0\vert^2 - r^2) + \e r^2.
$$
Observe that since $\f \leq q$ on $2 B$, we obtain

- if $x \in \partial B,$  $\f_{\de} (x) - q^\e (x) = \f_{\de} (x) - q (x) - \e r^2$, 

- If $x = x_0$, $\f_{\de} (x_0) -  q^{\e} (x_0) = \f_{\de} (x_0) - q (x_0) + \e r^2 $. 

\noindent Now $\f_{\de} (x_0) - q (x_0) \to \f (x_0) - q (x_0) = 0$ as $\de \to 0$, so that for $\de$ small enough, the function
$\f_{\de} (x) - q^{\e} (x)$ takes it maximum on $\bar B$ at some interior point $x_{\de} \in B$ and this maximum satisfies the inequality
\begin{equation}
\lim_{\de \to 0} \max_{\bar B} (\f_{\de} - q^{\e}) = \lim_{\de \to 0} (\f_{\de} (x_{\de}) -  q^{\e} (x_{\de})) \geq  \e r^2.
\label{eq:max-minoration}
\end{equation}
We claim that  $x_{\de} \to x_0$. Indeed 
\begin{eqnarray*}
\f_{\de} (x_{\de}) -  q^{\e} (x_{\de}) &= & \f_{\de} (x_{\de}) - q (x_{\de}) - 2 \e (\vert x_{\de} - x_0\vert^2 - r^2) - \e r^2 \\
& = & o (1) - 2 \e \vert x_{\de} - x_0\vert^2 + \e r^2.
\end{eqnarray*}
 If $x'_0$ is a limit point of the family $(x_{\de})$ in $\bar B$, then $\max_{\bar B} (\f_{\de} - q^{\e})$ converges to  $- 2 \e \vert x'_{0} - x_0\vert^2 + \e r^2$. 
 By the inequality (\ref{eq:max-minoration}), this limit is $\geq \e r^2$. Therefore $- 2 \e \vert x'_{0} - x_0\vert^2  \geq 0$, hence $x'_0 = x_0$ as claimed.
 
From the above properties we conclude that $dd^c \f_{\de} (x_{\de}) \leq dd^c q^{\e} (x_{\de})$, hence by inequality $(\ref{eq:conv-ineq})$ for $\de > 0$ small enough we get
$$
 (dd^c q^{\e} (x_{\de}))^n \geq  e^{\f_{\de} (x_{\de})}   f_{\de} \beta_n =  e^{\f_{\de} (x_{\de}) - q^{\e} (x_{\de})} e^{q^{\e} (x_{\de})}  f_{\de} (x_{\de}) \beta_n
$$
Now observe that  $\f_{\de} - q^{\e} = (\f_{\de} - q) + (q - q^{\e})$ and by Dini's lemma $\limsup_{\de \to 0} \max_{\bar B} (\f_{\de} - q) = \max_{\bar B} (\f - q) = 0$. Therefore
 $$
 \overline{\lim}_{\de \to 0} (\f_{\de} (x_{\de}) - q^{\e} (x_{\de})) \geq \underline{\lim}_{\de \to 0} \min_{\bar B}(q - q^{\e}) = 
 \min_{\bar B}(- 2 \e \vert x- x_0\vert^2 + \e r^2)
 = - \e r^2.
 $$
Since $q^{\e}$ converges in $C^2-$norm to the function $q$ we infer
$$
(dd^c  q^{\e} (x_{0}))^n \geq  e^{q (x_0) - 2 \e r^2} f (x_0) \beta_n.
$$
In the same way, as $\e \to 0$ we obtain  the required inequality $(dd^c q (x_0))^n \geq e^{\f (x_0)} f (x_0) \beta_n$, since $q (x_0) = \f (x_0)$.
\end{proof}

\begin{lem}\label{convol}
Let $u$ be a bounded plurisubharmonic function in a domain $\Omega \subset \C^n$ and $v = f \beta_n$ a continuous volume form with continuous density $f \geq 0$.
Assume that $\f$ satisfies 
$$
(dd^c \f)_{BT}^n \geq  e ^{\f} f \beta_n, \leqno (\star)
$$
in the  pluripotentiel sense in $\Omega$. 
Then  for $\de > 0$ small enough, the usual regularization by convolution $\f_{\de} := \f \star \chi_{\de}$ satisfies 
$$
(dd^c \f_{\de})_{BT}^n \geq e^{\f_{\de}} f_{\de} \beta_n, \ \ \text{with} \ f_{\de} (x) := \inf \{ \vert f (y)\vert ; \vert y - x \vert \leq \de,
$$
pointwise in $\Omega_{\de}$.
\end{lem}

\begin{proof}
This follows at least formally from the concavity of 
the function $A \longmapsto (\text {det} A)^{1 \slash n}$ 
on the convex cone of non negative hermitian matrices.
 However here $A$ will be a non negative hermitian matrix with 
Radon measure coefficients. In this context, Bedford and Taylor \cite{BT1} 
defined an operator $\Phi$ on all plurisuharmonic 
functions on $\Omega \subset \C^n$ extending
 the $n^{\text{th}}$ root of the determinant of the Levi 
form for smooth psh functions $u$. 
Namely let $u \in PSH (\Omega)$ and 
let  $dd^c u = \sum_{i,j} ( u_{j {\bar k}} + 
\sigma_{j \bar k}) d z_i \wedge d \bar z_j$ 
be the Lebesgue decomposition of the positive current $dd^c u$ on $\Omega$, 
where $(u_{j {\bar k}})$ is a hermitian 
positive matrix whose coefficients are 
$L^1$-loc functions on $\Omega$ and $\sigma_{j \bar k}$ are singular
 measures on $\Omega$. Using a general construction
 of Goffman and Serrin, they show that 
the following definition makes sense 
$$
\Phi (u) := \left(\text{det} (u_{j \bar k})\right)^{1 \slash n},
$$
as an absolutely continuous measure with 
$L^1_{loc}$ density with respect 
to the Lebesgue measure on $\Omega$.

It is easy to see that $\Phi$ is concave and positively homogenous.
This implies
$$
 \Phi (u \star \chi_{\de}) \geq \Phi (u) \star \chi_{\de}.
$$
in $\Omega_{\de}$.
We infer from $(\star)$ that
$$
 \Phi (u \star \chi_{\de}) \geq g_{\de} \beta_n,
$$
where $g := f^{1\slash n} \exp (\f\slash n)$ and $g_{\de} := g \star \chi_{\de}$ on $\Omega_{\de}$. The convexity of the exponential function now yields 
\begin{eqnarray*}
g_{\de} (x) & \geq & \inf_{\vert y - x\vert \leq \de} f^{1 \slash n} (y) \int \exp (u (x - y)\slash n) \chi_{\de} (y) d y \\ 
& \geq &   f_{\de}^{1 \slash n} (x) \exp ( u_{\de} (x) \slash n),
\end{eqnarray*}
 which implies that $(dd^c u \star \chi_{\de})^n \geq f_{\de} \exp(u_{\de})$, as claimed.
 \end{proof}

\section{The viscosity comparison principle for $(\omega+dd^c \f)^n=e^{\e\f} v$}

 We now set the basic frame for the viscosity approach to the equation
 $$
  (\omega+dd^c \f)^n =e^{\e \f} v,
 \leqno{(DMA^{\e}_v)}
 $$
 where $\omega$ is a closed smooth real $(1,1)$-form on a $n$-dimensional connected  
complex manifold $X$,
$v$ is a volume form with nonnegative continuous density and $\e \in \R_{+}$. 
Here the emphasis is on global properties.
 
 \smallskip

The global comparison principle lies at the heart of the viscosity approach. 
Once it is established, Perron's method
can be applied to produce viscosity solutions. Our main goal in this section is to establish the global comparison
principle for $(DMA^{\e}_v)$. 
We generally assume $X$ is compact (and $\e>0$): the structure of 
$(DMA_v^{\e})$ allows us to avoid any restrictive curvature 
assumption on $X$ (unlike e.g. in \cite{AFS}).

 \subsection{Definitions for the compact case}

 To fit in with the viscosity point of view, we rewrite the Monge-Amp\`ere equation as
$$
e^{\e \f}v-(\omega+ dd^c \f)^n=0 \leqno{(DMA_v^{\e})}
$$

Let $x\in X$.
If $\kappa \in \Lambda^{1,1} T_{x} X$ we define $\kappa_+^n$ to be $\kappa^n$ if $\kappa \ge 0$ and $0$ otherwise. 
For a technical reason, we will also consider a slight variant of $(DMA_v^{\e})$,

$$
e^{\e \f}v-(\omega+ dd^c \f)_+^n=0 \leqno{(DMA_v^{\e})_+}
$$

We let $PSH(X,\omega)$ denote the set of all $\omega$-plurisubharmonic ($\omega$-psh for short)
 functions on $X$: these are integrable  upper semi-continuous functions
$\f:X \rightarrow \R \cup \{-\infty\}$ such that $dd^c \f \geq -\omega$
in the sense of currents.

\smallskip

\begin{lem} \label{hyp}
Let $\Omega \subset X$ be an open subset and $z:\Omega\to \C^n$ be a holomorphic coordinate chart. 
Let $h$ be a smooth local potential for $\omega$ defined on $\Omega$.
Then $(DMA^{\e}_v)$ reduces in these $z$-coordinates to the scalar equation
$$
 e^{\e u} W - \det(u_{z\bar z})=0 \leqno{(DMA_{v|z}^{\e})}
$$
where $u=(\f+h)|_{\Omega} \circ z^{-1}$, $z_*v=e^{\e h_{| \Omega} \circ z^{-1}}W d\lambda$ and $\lambda$ is the Lebesgue measure on $z(\Omega)$. 

On the other hand, $(DMA^{\e}_v)_+$ reduces  to the scalar equation:
$$
 e^{\e u} W - \det(u_{z\bar z})_+=0 \leqno{(DMA_{v|z}^{\e})_+}
$$
\end{lem}

The proof is straightforward. 
Note in particular that conditions (1.2) p. 27 (ie.: degenerate ellipticity)  (2.11) p. 32 (properness), (2.18) p. 34
in \cite{IL} are satisfied by $(DMA_{v|z}^{\e})_+$. If $v>0$ and $\e>0$,  (2.17) p.33 is also satisfied,
so that we can apply the tools exposed in \cite{IL,CIL}.

\smallskip

\subsubsection{Subsolutions}

If $\f^{(2)}_x$ is the $2$-jet at $x\in X$ of a ${\mathcal C}^2$ real valued function $\f$
we set
$$
F_+(\f^{(2)}_x)=F^{\e}_{+,v}(\f_x)=e^{\e \f(x)}v_x-(\omega_x+ dd^c \f_x)_+^n .
$$
Recall the following definition from \cite{IL}. 

\begin{defi}
A subsolution of $(DMA^{\e}_v)_+$ is an upper semi-continuous function $\f: X\to \R \cup \{-\infty\}$ such that
$\f\not \equiv -\infty$
and the following property is satisfied:
if $x_0\in X$ and  $q \in {\mathcal C}^2$, defined in a neighborhood of $x_0$, is such that $\f(x_0)=q(x_0)$ and
$$
 \f-q \  \text{ has a local maximum at} \ x_0,
 $$ 
 then $F_+(q^{(2)}_{x_0})\le 0$. 
\end{defi}

Actually, this concept of a subsolution seems to be a bit too weak. 
It is not the same concept for $\epsilon=0$ as in section \ref{sec:local}
and does not behave well if $v=0$ since any usc function is then a viscosity subsolution of 
$(dd^c\f)^n_+=0$. It behaves well however if $v>0$. 
We introduce it nevertheless in order to be able to use the reference \cite{IL}.

We now introduce what we believe to be the right definition, which leads to a slightly stronger statement. 
If $\f^{(2)}_x$ is the $2$-jet at $x\in X$ of a ${\mathcal C}^2$ real valued function $\f$
we set
$$
F(\f^{(2)}_x)=F^{\e}_{v}(\f_x)=
\left\{
\begin{array}{ll}
e^{\e \f(x)} v_x-(\omega_x+ dd^c \f_x)^n & \text{ if } \omega+dd^c\f_x\ge 0 \\
+\infty & \text{otherwise}.
\end{array}
\right.
 $$
Recall the following definition from \cite{CIL}:

\begin{defi}
 A subsolution of $(DMA^{\e}_v)$ is an upper semi-continuous function $\f: X\to \R \cup \{-\infty\}$ such that
$\f\not \equiv -\infty$
and the following property is satisfied:
if $x_0\in X$ and  $q \in {\mathcal C}^2$, defined in a neighborhood of $x_0$, is such that $\f(x_0)=q(x_0)$ and
$$
 \f-q \  \text{ has a local maximum at} \ x_0,
 $$ 
 then $F(q^{(2)}_{x_0})\le 0$. 
 \end{defi}

\begin{rem}
The function $F^{\e}_{v}$ is lower semicontinuous and satisfies conditions (0.1) and (0.2)
in \cite{CIL}.
\end{rem}

Note that it is easy to compare subsolutions of $(DMA^{\e}_v)$ and $(DMA^{\e}_v)_+$:

\begin{lem}
Every subsolution $\f$ of $(DMA^{\e}_v)$ is  
a subsolution of $(DMA^{\e}_v)_+$, 
it is $\omega$-plurisubharmonic.

A  locally bounded usc function is $\omega$-psh and satisfies
  $(\omega+dd^c\f)^n_{BT} \ge e^{\epsilon \f} v$ iff it is a (viscosity) subsolution of $(DMA^{\e}_v)$ .

If $v>0$ subsolutions  of $(DMA^{\e}_v)_+$ are subsolutions of $(DMA^{\e}_v)$. 
\end{lem}
 
\begin{proof} Immediate consequence 
of the definitions, Theorem \ref{thm:visc=pluripot} and Proposition 
\ref{pro:visc=pluripot+tw}. One just has to choose a local potential $\rho$ such that
$dd^c\rho=\omega$ and set $\f'=\f+\rho$, $v'= e^{-\e \rho} v$ to apply the local results
of section \ref{sec:local}. 

\end{proof}

Actually, the discussion after Theorem \ref{thm:visc=pluripot}
fits well in the theory developped in \cite{BEGZ} and we get the:

\begin{coro}
Let $X$ be a compact K\"ahler manifold and $\omega$ a smooth closed  real $(1,1)$ form 
whose cohomology class $[\omega]$ is big. Let $\f$ be any $\omega$-plurisubharmonic function.
Then
$\f$ satisfies $(\omega+dd^c \f)^n \ge e^{\e \f} v$ in
 the vicosity sense iff $\langle (\omega+dd^c \f)^n \rangle \ge e^{\e \f} v$,
where $\langle (\omega+dd^c \f)^n \rangle$ is 
the non-pluripolar (pluripotential)  Monge-Amp\`ere measure of $\f$.
\end{coro}

\subsubsection{(Super)solutions}

\begin{defi}
A supersolution of $(DMA^{\e}_v)$ is a supersolution of $(DMA^{\e}_v)_+$, that is, a lower semicontinuous  function 
$\f: X\to \R \cup \{+\infty\}$ such that $\f\not \equiv +\infty$
and the following property is satisfied:
if $x_0\in X$ and  $q \in {\mathcal C}^2$, defined in a neighborhood of $x_0$, is such that $\f(x_0)=q(x_0)$ and
$
 \f-q \  \text{ has a local minimum at} \ x_0,
 $
 then $F_+(q^{(2)}_{x_0})\geq 0$. 
\end{defi}

\begin{defi}
A viscosity solution of $(DMA^{\e}_v)$ is a function that is both a sub-and a supersolution. In particular, viscosity solutions
are automatically continuous.

A pluripotential solution of $(DMA^{\e}_v)$ is an
 usc function $\f \in L^{\infty}\cap PSH(X,\omega)$
such that $(\omega+dd^c\f)^n_{BT}= e^{\e\f} v$.

Classical sub/supersolutions are ${\mathcal C}^2$ viscosity sub/supersolutions. 
\end{defi}

\subsection{The local comparison principle} %%%%2.2

\begin{defi} \text{ }

1) The local (viscosity) comparison principle for $(DMA^{\e}_v)$ is said to hold if the following holds true:
 let $\Omega\subset X$ be an open subset such that $\bar \Omega$ is biholomorphic to a bounded smooth strongly pseudoconvex domain in $\C^n$;
let  $\underline{u}$ (resp. $\overline{u}$) be a bounded subsolution (resp.supersolution)  of $(DMA^{\e}_v)$ in $\Omega$ satisfying
$$
 \limsup_{z\to \partial \Omega}  \left[ \underline{u}(z)-\overline{u}(z) \right] \le 0
$$ 
Then $\underline{u}\le \overline{u}$. 

\smallskip

2) The global (viscosity) comparison principle  for $(DMA^{\e}_v)$ is said to hold if $X$ is compact and the following holds true:
let $\underline{u}$ (resp. $\overline{u}$) be a bounded subsolution (resp. supersolution) of $(DMA^{\e}_v)$  in $X$.
Then $\underline{u}\le \overline{u}$. 
\end{defi}

We set the same definition with $(DMA^{\e}_v)_+$ in place of $(DMA^{\e}_v)$. 
Observe that $(DMA^{\e}_v)_+$ may have extra subsolutions, thus the comparison principle for $(DMA^{\e}_v)_+$
implies the comparison principle for $(DMA^{\e}_v)$. 

\medskip

The local viscosity comparison principle does not hold for $(DMA^0_{0})_+$. 
Indeed every usc function is a subsolution, the condition 
to be tested is actually empty. 
It is not clear whether it holds for $(DMA^{0}_0)$
since it is actually a  statement which differs 
substantially from the (pluripotential) comparison principle for the complex Monge-Amp\`ere 
equation of \cite{BT2}.

\begin{prop}\label{comp}
The local viscosity comparison principle for $(DMA^{\e}_v)$ holds if $\e>0$ and  $v>0$.
\end{prop}

\begin{proof}
The proposition actually follows from \cite{AFS}, corollary 4.8. 
We include  for the reader's convenience
a proof which is an adaptation of arguments 
in \cite{CIL}.

We may assume without loss of generality that $\e=1$. 
Let $\underline{u}$ be a bounded subsolution
and $\overline{u}$ be a bounded supersolution
 of $(DMA^{\e}_v)$ in some smoothly bounded
strongly pseudoconvex open set $\Omega$, 
such that $\underline{u} \leq \overline{u}$ on $\partial{\Omega}$.
Replacing first $\underline{u},\overline{u}$ by 
$\underline{u}-\delta,\overline{u}+ \delta$, we can assume that 
the inequality is strict and holds in a small neighborhood of $\partial \Omega$.

 As in the proof of Proposition \ref{pro:visc=pluripot+tw}, we regularize
 $\underline{u}$ and $\overline{u}$
 using their sup/inf 
convolutions.
Since $\underline{u},\overline{u}$ are bounded, multiplying by a small constant, we can assume that for $\a > 0$ small enough and $ x \in \Omega_{\a}$ , we have 
$$
\underline{u}^{\a}(x):=\sup_{y \in \Omega} \left\{ \underline{u}(y) -\frac{1}{2\a^2} |y-x|^2 \right\} = 
\sup_{\vert y - x\vert \leq \a } \left\{ \underline{u}(y) -\frac{1}{2\a^2} |y-x|^2 \right\}, 
$$
and
$$
\overline{u}_{\a}(x):=\inf_{y \in \Omega_{\a}} \left\{ \overline{u}(y) +\frac{1}{2\a^2} |y-x|^2 \right\}
= \inf_{\vert y - x\vert \leq \a } \left\{ \overline{u}(y) +\frac{1}{2\a^2} |y-x|^2 \right\}.
$$

Then for $\a > 0$ small enough $\underline{u}^{\a}(x) \leq \overline{u}_{\a}(x)$ near the boundary of $\Omega_{\a}$. Observe that, if we set  
$M_{\a}:=\sup_{\overline{\Omega}_{\a}} [\underline{u}^{\a}-\overline{u}_{\a}]$, then 
$$
\liminf_{\a \rightarrow 0^+} M_\a \geq  \sup_{\Omega} [\underline{u}-\overline{u}].
$$

Arguing by contradiction, assume that 
$\sup_{\Omega} [\underline{u} - \overline{u}] > 0$. 
Then for $\a > 0$ small enough, 
the supremum $M_\a$ is $>0$ and then it is attained at some point $x_{\a} \in \Omega_{\a}$. 

The function $\underline{u}^{\a}$ is semi-convex and $ \overline{u}_{\a}$ is
semi-concave. 
In particular they are twice differentiable
 almost everywhere on $\Omega_{\a}$ by a theorem of Alexandrov \cite{Ale} 
(see also \cite{CIL}) in  the following sense:

\begin{defi} \label{w2diff}
 A real valued function $u$ 
defined an open set $\Omega \subset \C^n$ 
is twice differentiable  at almost every point  $z_0 \in  
 \Omega$ if and only if for every point  $z_0 \in  
 \Omega$ outside a Borel set of Lebesgue measure 
$0$ in $\Omega$, there exists a quadratic form $Q_{z_0} u$ on $\R^{2 n}$,
 whose polar symetric  bilinear form will be  denoted by $D^2 u (z_0)$, such that for 
 any $ \xi \in \R^{2 n}$ with $\vert \xi \vert << 1$, we have
 
 \begin{equation}
  u (z_0 + \xi) = u (z_0) + D u (z_0) \cdot \xi + (1\slash 2) D^2 u (z_0)\cdot (\xi,\xi) + o (\vert \xi\vert^2). \ \ 
  \end{equation}

\end{defi}

We first deduce a contradiction 
under the unrealistic assumption that
 they are twice differentiable at the point $x_{\a}$. Then by the classical maximum principle we have 
$$
 D^2 \underline{u}^{\a} (x_\a) \leq D^2 \overline{u}_{\a} (x_\a),
$$
in the sense of quadratic forms on $\R^{2n}$.
Applying this inequality for vectors of the form $(Z,Z)$ and
 $(i Z, i Z)$ and adding we get the same inequality for Levi forms on $\C^n$, i.e.:

$$
  0 \leq dd^c \underline{u}^{\a} (x_\a) \leq dd^c \overline{u}_{\a}(x_\a),
$$
where the first inequality follows from the fact that $\underline{u}^{\a}$ is plurisubharmonic on $\Omega_\a$ since $\underline{u}$ is.
 From this inequality between non negative  hermitian forms on $\C^n$,
 it follows that the same inequality holds between their determinants i.e.:
$$
 (dd^c \underline{u}^{\a})^n (x_\a) \leq (dd^c \overline{u}_{\a})^n(x_\a).
$$
We know that 
$$
(dd^c \overline{u}_{\a})^n (x_\a) \leq e^{\overline{u}_{\a} (x_\a)} f^{\a} (x_\a) \beta_n,
$$
where $f_{\a}$ increases pointwise towards $f$, with $v=f \beta_n$,
and 
$$
(dd^c \underline{u}^{\a})^n (x_\a) \geq e^{\underline{u}^{\a} (x_\a)} f_{\a} (x_\a) \beta_n,
$$
where $f^{\a}$ decreases towards $f$ pointwise. Therefore we have for $\a$ small enough,
\begin{equation}
e^{\underline{u}^{\a} (x_\a)} f_{\a} (x_\a) \leq e^{\overline{u}_{\a} (x_\a)} f^{\a} (x_\a).
\label{eq:FINEQ1}
\end{equation}
From this inequality we deduce immediately that 
$$
\sup_{\Omega} [\underline{u}-\overline{u}] \leq \overline{\lim}_{\a \rightarrow 0} M_{\a} \leq 0=\lim_{\a \rightarrow 0} \log \frac{f^{\a}(x_\a)}{f_{\a}(x_{\a})},
$$
which contradicts our assumption that $\sup_{\Omega} [\underline{u}-\overline{u}] > 0$.

\smallskip

When $\underline{u}^{\a}, \overline{u}_{\a}$ 
are not twice differentiable at point $x_{\a}$ for a fixed $\a >0$ 
small enough, we  prove that the inequality (\ref{eq:FINEQ1})
 is still valid by approximating $x_\a$ by a 
sequence of points where 
the functions are twice differentiable and not 
far from attaining their maximum at that points. 
 For each $k \in \N^*$, the
semiconvex function $\underline{u}^{\a} - \overline{u}_{\a} - 
(1 \slash 2 k) \vert x - x_\a\vert^2 $ attains a strict maximum at $x_\a$.
 By  Jensen's lemma 
(\cite{Jen}, see also \cite{CIL} Lemma A.3 p. 60),
there exists a sequence $(p_k)_{k \geq 1}$ of  vectors converging to 
$0$ in $\R^n$ and a sequence of points $(y_k)$ converging to $x_\a$ in $\Omega_\a$ such that  the functions $\underline{u}^{\a}$ and $ \overline{u}_{\a}$ are twice differentiable at $y_k$ and, if we set  $q_k (x) = (1 \slash 2 k) \vert x - x_\a\vert^2 + <p_k, x>$, the function $\underline{u}^{\a} - \overline{u}_{\a} - q_k$ attains its maximum on $\Omega_\a$ at the point $y_k$. 

 Applying the classical maximum principle for fixed $\a$ at each point $y_k$ we get
 $$
 D^2  \underline{u}^{\a} (y_k) \leq D^2\overline{u}_{\a} (y_k) + (1 \slash k) I_n,
 $$
 in the sense of quadratic forms on $\R^{2 n}$.
  As before we obtain the following inequalities between Levi forms
\begin{equation} \label{eq:Maxp}
 0 \leq dd^c \underline{u}^{\a} (y_k) \leq dd^c \overline{u}_{\a} (y_k) + (1 \slash k) dd^c \vert x\vert^2 ,
\end{equation}
in the sense of positive hermitian forms on $\C^n$, where the first inequality follows from the fact that $\underline{u}^{\a}$ is plurisubharmonic on $\Omega_\a$.
 The inequality (\ref{eq:Maxp}) between positive hermitian forms implies the same inequality between their determinants, so that
 
\begin{equation} \label{eq:Det1}
 (dd^c \underline{u}^{\a}(y_k))^n  \leq (dd^c \overline{u}_{\a} (y_k) + (1 \slash k) dd^c \vert x\vert^2)^n
\end{equation}

Recall that  $ \overline{u}_{\a} - (1 \slash 2 \a^2)\vert x \vert^2$ is
  concave on $\Omega_\a$, hence:
$$
  D^2 \overline{u}_{\a} (y_k) \leq  (1 \slash 2 \a^2) I_n,
$$
 in the sense of quadratic forms on $\R^{2 n}$. Therefore:
 \begin{equation} \label{eq:Lev}
  dd^c \overline{u}_{\a} (y_k) \leq  (1 \slash  2\a^2) dd^c \vert x\vert^2.
 \end{equation}
 From  (\ref{eq:Maxp}) and (\ref{eq:Lev}), it follows that, $\alpha$ being fixed :
 \begin{equation} \label{eq:Det2}
 (dd^c \overline{u}_{\a} (y_k) + (1 \slash k) dd^c \vert x\vert^2)^n = (dd^c \overline{u}_{\a}(y_k))^n + O (1\slash k).
 \end{equation}

 We know by definition of subsolutions and supersolutions that 
 \begin{equation} \label{eq:Sub-Sup}
 (dd^c \underline{u}^{\a}(y_k))^n  \geq e^{\underline{u}^{\a} (y_k)} f_{\a} (y_k) \beta_n, \ \ (dd^c \overline{u}_{\a}(y_k))^n  \leq  
 e^{\overline{u}_{\a} (y_k)} f^{\a} (y_k) \beta_n.
 \end{equation}
 Therefore from the inequalities (\ref{eq:Det1}), (\ref{eq:Det2}) and (\ref{eq:Sub-Sup}), it follows that for any $k \geq 1$ we have

$$
e^{\underline{u}^{\a} (y_k)} f_{\a} (y_k) \leq e^{\overline{u}_{\a} (y_k)} f^{\a} (y_k) +  O (1\slash k),
$$
which implies the inequality (\ref{eq:FINEQ1}) as $k \to + \infty$.
The same argument as above then gives a contradiction. 
\end{proof}

\subsection{The global viscosity comparison principle}

The global comparison principle can be deduced from  \cite{AFS} when
$X$ carries a K\"ahler metric with positive sectional curvature. 
This global curvature
assumption is very strong: as explained in the introduction
 \cite{Mok} reduces us to a situation
where one can regularize $\omega$-psh functions with no loss of positivity \cite{Dem2},
so that the viscosity approach is not needed to achieve continuity (see \cite{EGZ1}).
On the other hand \cite{AFS} considers very general degenerate elliptic equations
whereas we are considering a rather restricted class of complex
 Monge-Amp\`ere equations.

In the general case, neither \cite{AFS} nor \cite{HL} allows to establish a global 
comparison principle even in the simplest case that we now consider:

\begin{prop} \label{ay} The global comparison
principle holds when the cohomology class of $\omega$ is K\"ahler
and $v$ is continuous and positive.
\end{prop}

\begin{proof}
We can assume without loss of generality that $\e=1$.

Assume first $v>0$ and smooth. By \cite{Aub,Yau}, there is $\f_Y\in C^{2}(M)$ a 
classical solution of 
$(DMA^{1}_v)$. If $\underline{u}$ is a subsolution then $\underline{u}\le \f_Y$,
as follows from Lemma \ref{lem:classical} below. 
Similarly, if $\overline{u}$ is a supersolution $\overline{u}\ge \f$. 
If $v>0$ is merely continuous, fix $\delta>0$. Then, if $\underline{u}$ is a subsolution of 
$(DMA^{1}_v)$, 
$\underline{u}-\delta$ is a subsolution of $(DMA^{1}_{e^{\delta}v})$. 
Similarly, if $\overline{u}$ is a supersolution of $(DMA^{1}_v)$, $\overline{u}+\delta$ is a supersolution of $(DMA^{1}_{e^{-\delta}v})$. 
Choose $v^*$ a smooth volume form
such that $e^{-\delta}v<v^*<e^{\delta}v$. Then $\underline{u}-\delta$ is a subsolution
of $(DMA^{1}_{v^*})$ and $\overline{u}+\delta$ a supersolution. Hence,
$\underline{u}-\delta \le \overline{u}+\delta$. Letting $\delta \to 0$, we conclude the proof. 
\end{proof}

\begin{lem} \label{lem:classical}
Assume $v>0$. 
Let $\underline{u}$  be a bounded subsolution of $(DMA^{\e}_v)$ on $X$. If  $\overline{u}$  is a ${\mathcal C}^2$ supersolution on $X$,
then $\underline{u}\le \overline{u}$. 
\end{lem}

Note in particular that if $(DMA^{\e}_v)$ has a classical solution then it dominates (resp. minorates) 
every subsolution 
(resp. supersolution), hence the global viscosity principle holds.

\begin{proof}
 If $\overline{u}$ is classical, the fact that $\underline{u} \le \overline{u}$ is a trivial consequence of the definition of subsolution at a maximum of $\underline{u} - \overline{u}$. Indeed let $x_0\in X$ such that
$\underline{u}(x_0) - \overline{u}(x_0)= \max _X (\underline{u} - \overline{u})=m$. 
Use $q=\overline{u}+m$
as a test function in the definition of $\underline{u}$ being a viscosity subsolution to deduce: 
$$ (\omega+dd^c\overline{u})^n_{x_0} \ge e^{\overline{u}(x_0)+m}v. 
$$
On the other hand, $\overline{u}$ being a classical supersolution, we have
$$(\omega+dd^c\overline{u})^n_{x_0} \le e^{\overline{u}(x_0)}v. 
$$
Hence $m\le 0$. 
\end{proof}

\smallskip

The above remarks have only academic interest since we use existence of a classical solution to deduce
a comparison principle whose main consequence is existence of a viscosity solution (cf. infra).
 We need to establish a honest global comparison principle that will allow
us to produce solutions without invoking \cite{Aub,Yau}.
We now come to this result:

\begin{theo}\label{qcp}
The global viscosity comparison principle for $(DMA^{\e}_v)$ holds, 
provided $\omega$ is a closed real $(1,1)$-form, $v>0$, $\e>0$ and $X$ is compact
\footnote{We do NOT assume that $X$ is K\"ahler. However, this statement seems to be useless outside the Fujiki class.}.
\end{theo}

Since $v>0$, the subsolutions of $(DMA^{\e}_v)_+$ are those of  $(DMA^{\e}_v)$ hence this could be 
stated as the global viscosity comparison principle for $(DMA^{\e}_v)_+$. 
\begin{proof}
As above, we assume $\e=1$.
 Let $u^*$  be a bounded supersolution and $u_*$ be a bounded subsolution. We choose $C>0$ such that 
both are $\le C/1000$ in $L^{\infty}$-norm. Since $u_*-u^*$
is uppersemicontinuous on the compact manifold $M$, it follows that its maximum is achieved at some point $\hat x_1 \in M$. 
 Choose complex coordinates $(z^1, \ldots, z^n)$ near $\hat x_1$
defining a biholomorphism identifying an open neighborhood of $\hat{x}_1$ to the complex ball $B(0,4)$ of radius $4$ sending $\hat{x}_1$ to zero. 

Using a partition of unity,
 construct a riemannian metric on $M$ which coincides with the
 flat K\"ahler metric $\sum_k\frac{\sqrt{-1}}{2} dz^k \wedge d\bar z^k$ on the ball of center
$0$ and radius $3$. For $(x,y)\in M\times M$ define $d(x,y)$ to be the riemannian distance function. The  continuous function $d^2$ is of class ${\mathcal C}^2$ near the diagonal and $>0$ outside the diagonal $\Delta\subset M ^2$.

Next, construct a smooth non negative function $\f_1$ on $M\times M$ by the following formula:
$$ \f_1(x,y)=\chi (x,y). \sum_{i=1}^n |z^i(x)-z^i(y)|^{2n},$$
where $\chi$ smooth non negative cut off function with $1\ge \chi \ge 0$,
$\chi\equiv 1$ on $B(0,2)^2$ $\chi=0$ near $\partial B(0,3)^2$. 

Finally, consider  a second smooth function on $M\times M$ with 
$\f_2|_{B(0,1)^2} <-1$,  $\f_2|_{M^2-B(0,2)^2} >3C$.

Choose $1\gg \eta >0$ such that  $-\eta$ is a regular value of both 
$\f_2$ and $\f_2|_{\Delta}$.

We perform convolution of $(\xi,\xi')\mapsto \max(\xi, \xi')$ by a smooth semipositive function $\rho$ such that $B_{\R^2}(0,\eta) =\{\rho>0\}$ and get a smooth function on $\R^2$
$\max_{\eta}$ such that: 
\begin{itemize}
\item $\max_{\eta} (\xi, \xi')=\max(\xi, \xi')$ if $|\xi-\xi'|\ge\eta$, 
\item $\max_{\eta} (\xi, \xi')>\max(\xi, \xi')$ if $|\xi-\xi'|<\eta$. 
\end{itemize}

We define $\f_3\in {\mathcal C}^{\infty}(M^2,\R)$ to be $\f_3=\max_{\eta}(\f_1,\f_2)$. Observe that:
\begin{itemize}
\item $\f_3\ge 0$,
\item $\f_3^{-1}(0)= \Delta \cap \{ \f_2 \le -\eta \}$,
\item $\f_3|_{M^2-B(0,2)^2} >3C$.
\end{itemize}

We define $h_{\omega}\in {\mathcal C}^2(\overline{B(0,4)}, \R)$ to be a local potential smooth up to the boundary  for $\omega$ and extend it smoothly to $M$.
We may without losing generality assume that $\|\bar h_{\omega} \|_{\infty} < C/10$.
In particular $dd^c h_{\omega}= \omega$ and $w_*=u_*+h_{\omega}$ is a viscosity subsolution
of 
$$
(dd^c \f)^n =e^{\f} W
\text{ in } B(0,4)
$$ 
with $W$ positive and continuous. 
On the other hand $w^*= u^*+h_{\omega}$ is a viscosity supersolution of the same equation.

Now fix $\alpha>0$. Consider $(x_{\alpha},y_{\alpha}) \in M^2$ such that:
\begin{eqnarray*}
M_{\alpha}&=&\sup_{(x,y)\in \overline{B(0,4)}^2} w_*(x)-w^*(y) -\f_3(x,y)-\frac{1}{2}\alpha d^2(x,y)\\
&=&w_*(x_{\alpha})-w^*(y_{\alpha}) -\f_3(x_{\alpha},y_{\alpha})-\frac{1}{2}\alpha d^2(x_{\alpha},y_{\alpha}).
\end{eqnarray*}

The sup is achieved since we are maximizing an usc function. We also have, taking into
account that $\phi_3(\hat x_1, \hat x_1)=0$:
$$ 
2C+ C/5 \ge M_{\alpha} \ge w_*(\hat x_1)-w^*(\hat x_1)\ge 0. 
$$
By construction, we see that $(x_{\alpha}, y_{\alpha}) \in B(0,2)^2$. 

\smallskip

Using \cite[Proposition 3.7]{CIL}, we deduce the:

\begin{lem}\label{ineq}
 We have $\displaystyle{\lim_{\alpha\to\infty}} \alpha d^2(x_{\alpha}, y_{\alpha})=0$. Every limit
point $(\hat x, \hat y)$ of $(x_{\alpha}, y_{\alpha})$ 
satisfies $\hat x= \hat y$, $\hat x \in \Delta \cap \{  \f_2 \le -\eta \}$
and 
\begin{eqnarray*}
w_*(\hat x)-w^*(\hat x)&=&u_*(\hat x)-u^*(\hat x) \\
&=&\max_{x\in \overline{B(0,4)}} w_*(x)-w^*(x)-\f_3(x,x)\\
&=&\max_{x\in M^2} u_*(x)-u^*(x)-\f_3(x,x)\\
&=& u_*(\hat x_1)-u^*(\hat x_1) \\
&=& w_*(\hat x_1)-w^*(\hat x_1) \\
\liminf_{\alpha \to +\infty} w_*(x_{\alpha})-w^*(y_{\alpha})&\ge & w_*(\hat x_1)-w^*(\hat x_1)
\end{eqnarray*}
\end{lem}

Next, we use \cite[Theorem 3.2]{CIL} with
 $u_1=w_*$, $u_2=-w^*$, $\f=\frac{1}{2}\alpha d^2 +\f_3$. 
For $\alpha \gg 1$, everything is localized to $B(0,2)$ hence
$d$ reduces to  the euclidean distance function.
Using the usual formula for the first and second derivatives 
of its square,   
we get the following:

\begin{lem}\label{main}
$\forall \e >0$, we can find $(p_*, X_*), (p^*, X^*)\in \C^n \times Sym_{\R}^2 (\C^n)$
s.t.
\begin{enumerate}
 \item $(p_*, X_*)\in \overline{J^{2+} }w_*(x_{\alpha})$,
\item  $(-p^*, -X^*)\in \overline{J^{2-}}  w^*(y_{\alpha})$,
\item The block diagonal matrix with entries $(X_*,-X^*)$ satisfies:
$$
-(\e^{-1}+ \| A \| ) I \le 
\left(
\begin{array}{cc}
X_* & 0 \\
0 & -X^*
\end{array}
\right)
\le A+\e A^2, 
$$
where $A=D^2\f(x_{\alpha}, y_{\alpha})$, i.e.
$$A =\alpha
\left(
\begin{array}{cc}
I &  -I\\
-I &  I
\end{array}
\right) +D^2\f_3 (x_{\alpha}, y_{\alpha})$$
and $\| A \|$ is the spectral radius of $A$ (maximum of the absolute values for the eigenvalues of this symmetric
matrix). 
\end{enumerate}
\end{lem}

By construction, the Taylor series of $\f_3$ at any point in
 $\Delta \cap \{  \f_2 < -\eta \}$ vanishes up to order $2n$. By transversality, 
$\Delta \cap \{  \f_2 < -\eta \}$
is dense in $\Delta \cap \{  \f_2 \le -\eta \}$, and this
 Taylor series  vanishes up to order $2n$ on $\Delta \cap \{  \f_2 \le -\eta \}$.
In particular, $$ D^2\f_3 (x_{\alpha}, y_{\alpha}) =O(d(x_{\alpha}, y_{\alpha})^{2n})
=o(\alpha^{-n}).$$
This implies $\|A \|\simeq\alpha$. We chose $\alpha^{-1}= \e$ and deduce
$$
-(2\alpha ) I \le 
\left(
\begin{array}{cc}
X_* & 0 \\
0 & -X^*
\end{array}
\right)
\le 3\alpha \left(
\begin{array}{cc}
I &  -I\\
-I &  I
\end{array}
\right) + o(\alpha^{-n})
$$

Looking at the upper and lower diagonal terms we deduce that the eigenvalues
of $X_*, X^*$ are $O(\alpha)$. Evaluating the inequality on vectors of the form 
$(Z,Z)$ we deduce from the $\le$ that the eigenvalues
of $X_{*}-X^{*}$ are $ o(\alpha^{-n})$.

\smallskip

Fix $X\in Sym_{\R}^2 (\C^n)$ and denote by $X^{1,1}$ its $(1,1)$-part. It is a hermitian matrix. 
 Obviously the eigenvalues of $X_*^{1,1}, X^{*1,1}$ are $O(\alpha)$ and those
of $X_{*}^{1,1}-X^{*1,1}$ are $o(\alpha^{-n})$. 
Since $(p_*, X_*)\in \overline{J^{2+} }w_*(x_{\alpha})$ 
we deduce from the definition of viscosity solutions 
that $X_*^{1,1}$ is positive definite and that the product of its $n$
eigenvalues is $\ge c>0$ uniformly in $\alpha$. In particular its 
smallest eigenvalue
is $\ge c\alpha^{-n+1}$. The relation $ X_{*}^{1,1}+o(\alpha^{-n})\le X^{*1,1}$
forces $X^{*1,1} >0$ and $det(X^{*1,1})/det(X_{*}^{1,1})\ge 1 +o(\alpha^{-1})$. 

Now,  since $(p_*, X_*)\in \overline{J^{2+} }w_*(x_{\alpha})$ and  $(-p^*, -X^*)\in \overline{J^{2-}}  w^*(y_{\alpha})$ we get by definition of viscosity solutions:
$$
\frac{det(X^{*1,1})}{det(X_{*}^{1,1})}\le \frac{
 e^{w^*(y_{\alpha})} W(y_{\alpha})}{e^{w_*(x_{\alpha})} W(x_{\alpha})}
$$
Upon passing to the superior limit as $\alpha\to +\infty$,
 we get $1\le e^{\limsup w^*(y_{\alpha})-w_*(x_{\alpha}) }$. Taking
lemma \ref{ineq} into account
$w^*(\hat x)\ge w_*(\hat x)$ thus $u^*(\hat x)\ge u_*(\hat x)$.
\end{proof}

\begin{rem}
 The miracle with the complex Monge Amp\`ere equation 
we are studying is that
 the equation does not depend on the gradient
 in complex coordinates. In fact, it  takes the 
form $F(X)-f(x)=0$. The localisation technique would fail without this structural feature. 
\end{rem}

\subsection{Perron's method}

Once the global comparison principle holds, one easily constructs continuous (viscosity=pluripotential) solutions
by Perron's method as we now explain.

\begin{theo} \label{perron}
Assume the global comparison principle holds for $(DMA^1_v)$ and that $(DMA^1_v)$ has a bounded subsolution $\underline{u}$
and a bounded supersolution $\overline{u}$. 
Then, 
$$
\f=\sup\{ w \, | \,  \underline{u} \le w
\le \overline{u} \ \text{and} \ w \ \text{is a viscosity subsolution of } (DMA^{1}_v) \} 
$$
is the unique viscosity solution of $(DMA^1_v)$. 

In particular, it is a continuous $\omega$-plurisubharmonic function. Moreover $\f$ is also a solution of $(DMA^1_v)$
in the pluripotential sense.
\end{theo}

\begin{proof} See
\cite{CIL} p. 22-24, with a grain of salt. Indeed, lemma 4.2 there implies that the upper 
enveloppe $\f$ of the subsolutions of $(DMA^1_v)$ is a 
subsolution of $(DMA^1_v)$ since $F$ is lsc. Hence $\f$ is a subsolution of $(DMA^1_v)_+$. 

The trick is now to consider its lsc enveloppe $\f_*$. We are going
to show that it is  a supersolution of $(DMA_v^1)$:
otherwise we find $x_0\in X$ and  $q$ a ${\mathcal C}^2$ function such that $\f_*-q$ 
has zero as a local minimum at $x_0$ and  $F_+(q^{(2)}_x )<0$.  
This forces $v_{x_0}>0$. 

Then proceeding as in loc.cit. p.24 
we can construct a subsolution $U$ such that $U(x_1)>\f(x_1)$ for some $x_1\in X$. 

This contradiction leads to the conclusion that $\f_*$ is a supersolution and, by the viscosity 
comparison principle, that $\f_*\ge \f$. Since $\f=\f^*\ge\f_*$ it follows that $\f=\f_*=\f^*$ 
is a continuous viscosity solution.

For the reader's convenience, we briefly summarize the construction of $U$. Let $(z^1, .., z^n)$ be a coordinate 
system centered at $x_0$ giving a local isomorphism with the complex unit ball and assume $v>0$ 
on this complex ball neighborhood. Then, for $\gamma, \delta, r>0$
small enough $q_{\gamma,\delta}= q +\delta-\gamma \|z\|^2$ satisfies $F_+(q_{\gamma,\delta}^{(2)} )<0$
for $\|z(x)\|\le r$.

Chose $\delta=(\gamma r^2)/8$, $r>0$ small enough. Since $\f_*(x)-q (x)\ge 0$ for $\|z (x) \|\le r$
we have $\f(x)\ge \f_*(x) > q_{\gamma,\delta}(x)$ if $r/2 \le \|z (x) \|\le r$. It follows that
$U$ defined by $$U(x)=\max(\f(x), q_{\delta, \gamma}(x))$$ if $\|z(x)\|\le r$ and $U(x)=\f(x)$ otherwise
is a subsolution of $(DMA_v^1)_+$ and in fact of $(DMA_v^1)$ since we may assume that $v>0$ on the
relevant part of $X$. Chose a sequence $(x_n)$ converging to $x_0$ so that $\f(x_n)\to \f_*(x_0)$. 
Then $q_{\gamma,\delta}(x_n) \to \f_*(x_0) +\delta$. Hence, for $n\gg 0$, $U(x_n)=q_{\gamma,\delta}(x_n)>\f(x_n)$.
 
It remains to be seen 
 that $\f$ is also a solution of $(DMA^1_v)$ in the pluripotential sense. 
It follows from the previous argument in pluripotential theory. 

In fact, since $\f$ is a viscosity subsolution, we know that 
$(\omega+dd^c\f)^n_{BT}\ge e^{\epsilon \f}v$ by Proposition \ref{pro:visc=pluripot+tw}. 
Now argue by  contradiction.
Namely choose $B \subset X$ a ball on which 
$(\omega+dd^c\f)^n_{BT}\not = e^{\epsilon \f }v$.
Solve a Dirichlet  problem to get a continuous psh function $\psi$ on  
$\bar B$ with $(\omega+dd^c \psi)^n_{BT}\not = e^{\epsilon \psi}v$
and $\psi=\f$ on $\partial B$. The Bedford Taylor comparison principle
gives $\psi \ge \f$ and $\psi\not=\f$ by hypothesis.
Also $\psi$ is a viscosity subsolution.  For $t>0$ small
enough $\f_0=\max( \f, \psi-t)$ is another viscosity subsolution
with $\f_0 > f$ on an open subset. This is contradiction to the definition
of $\f$ as an envelope. 
\end{proof}

In situations where the global comparison principle is not available, one can always
use the following substitute (this natural idea is used in the recent work \cite{HL}).

\begin{prop} \label{perron2}
Assume  that $(DMA^1_v)$ has a bounded subsolution $\underline{u}$
and a classical supersolution $\overline{u}$. 
Then, 
$$
\f=\sup\{ w \, | \,  \underline{u} \le w
\le \overline{u} \ \text{and} \ w \ \text{is a subsolution of } (DMA^{1}_v) \} 
$$
is the unique maximal viscosity subsolution of $(DMA^1_v)$. 

In particular, it is a bounded $\omega$-plurisubharmonic function.
\end{prop}

\begin{rem}
Assume $X$ is a complex projective manifold such that $K_X$ is ample. \
Let $\omega>0$ be a K\"ahler representative of $[K_X]$ and $v$
a volume form with $Ric(v)=-\omega$. Then the Monge-Amp\`ere equation 
$ (\omega+dd^c\f)^n =e^{\f} v$ satisfies all the hypotheses 
of Theorem \ref{perron} and has a unique viscosity solution $\f$. 
On the other hand, the Aubin-Yau theorem \cite{Aub,Yau} 
implies that it has a unique smooth solution $\f_{KE}$ (and $\omega+dd^c\f_{KE}$ is 
the canonical K\"ahler-Einstein metric on $X$). 
Uniqueness of the viscosity solution implies $\f=\f_{KE}$
hence the potential of the canonical KE metric on $X$ is the 
envelope of the (viscosity=pluripotential ) 
subsolutions to  $ (\omega+dd^c\f)^n =e^{\f} v$. 
\end{rem}

In the global case when $\e=0$, i.e. for $(\omega+dd^c \f)^n=v$ on a compact K\"ahler manifold, 
  classical strict sub/supersolutions do not exist and using directly  the Perron method seems doomed to failure.

\section {Regularity of potentials of singular K-E metrics}

In this section we apply the viscosity approach to show that the canonical 
singular K\"ahler-Einstein metrics
constructed in \cite{EGZ1} have continuous potentials.

\subsection{Manifolds of general type}

Assume $X$ is compact K\"ahler and $v$ is a continuous volume form with semi-positive density. 
Fix $\beta$ a K\"ahler form on $X$. We consider the following condition
on  (the cohomology class of) $\omega$: 
$$  
\exists \eta >0 \ \exists \psi \in L^{\infty}\cap PSH(X,\omega), \ (\omega + dd^c \psi)^n 
\ge\eta \beta^n.\leqno{(\dagger)}
$$

When $X$ is a compact K\"ahler manifold, $\omega$ is a semipositive $(1,1)$-form with 
$\int_{X}\omega^n>0$ then $(X,\omega)$ satisfies $(\dagger)$, as follows from \cite{EGZ1,BEGZ}.
However the latter articles rely on \cite{Aub,Yau} and we will show in the proof of 
Theorem \ref{thm:exp} how to check (\dag) directly, thus providing a new approach
to the ``continuous Aubin-Yau theorem''.

\smallskip

Note that the inequality can be interpreted in the pluripotential or viscosity sense since
these agree by Theorem \ref{thm:visc=pluripot}.  

\begin{lem} Assume $(\dag)$ is satisfied and $v$ has positive density.

If $c>>0$ is large enough, then $\psi-c$ is a subsolution of $(DMA^1_v)$ and 
the constant function $\f=c$ is a supersolution of $(DMA^1_v)$. 
\end{lem}

\begin{proof}
Existence of $\psi$ is needed for the subsolution whereas the supersolution exists under the condition that 
$\exists C>0$ such that $\omega^n \le C.v$,
which follows here from our assumption that $v$ has positive density.
\end{proof}

\begin{coro} \label{cor:dagOK}
When $(\dag)$ is satisfied and $v$ is positive,
$(DMA^1_v)$ has a unique viscosity solution $\f$, 
which is also the unique solution in the pluripotential sense.
\end{coro}

\begin{proof}
Indeed the global comparison principle holds  and Theorem \ref{perron} enables to conclude. 
\end{proof}

 We are now ready to establish that the (pluripotential) solutions of some Monge-Amp\`ere
 equations constructed in \cite{EGZ1} are continuous:

\begin{theo} \label{thm:exp}
Assume $X$ is a compact K\"ahler manifold, $\omega$ is a semipositive $(1,1)$-form with 
$\int_{X}\omega^n>0$  and $v$ is a semi-positive  continuous 
probability measure on $X$. Then $(\dag)$ is satisfied and there exists a unique
continuous $\omega$-plurisubharmonic function $\f$ which is
the viscosity (equivalently pluripotential) solution to the 
degenerate complex Monge-Amp\`ere equation
$$
(\omega+dd^c\f)^n =e^{\f}v 
$$
\end{theo}

\begin{coro}
The function  $\f_P\in L^{\infty} \cap PSH(X,\omega)$ such that 
$$
(\omega+dd^c\f_P)^n =e^{\f_P}v 
$$
in the pluripotential sense constructed in \cite{EGZ1} Theorem 4.1 
is a viscosity solution, hence it is continuous. 
\end{coro}

\begin{proof} 
Observe that $(\dag)$ is obviously satisfied when the cohomology class
of $\omega$ is K\"ahler.
If moreover $v$ has positive density, the result is an immediate consequence of 
Corollary \ref{cor:dagOK} together with the unicity statement \cite{EGZ1} proposition 4.3.  

We treat the general case by approximation.
We first still assume that $v$ is positive but the cohomology class $\{\omega \}$ is 
now merely semi-positive and big (i.e. $\int_X \omega^n>0$). This is a situation considered in
\cite{EGZ1} where it is shown that $(\dag)$ holds, however we would like to make clear
that the proof is independent of \cite{Yau} so we (re)produce the argument.
By the above there exists, for each $0<\e \leq 1$, a unique continuous $(\omega+\e \beta)$-psh function
$u_{\e}$ such that
$$
(\omega+\e \beta+dd^c u_{\e})^n=e^{u_{\e}} v.
$$

We first observe that $(u_{\e})$ is relatively compact in $L^1(X)$.
By \cite{GZ1}, this is equivalent to checking that $\sup_X u_{\e}$ is bounded, as $\e \searrow 0^+$.
Note that
$$
e^{\sup_X u_{\e}} \geq \frac{\int_X \omega^n}{v(X)}=\int_X \omega^n
$$
hence $\sup_X u_{\e}$ is uniformly bounded from below. Set $w_{\e}:=u_{\e}-\sup_X u_{\e}$.
This is a relatively compact family of $(\omega+\beta)$-psh functions, hence there
exists $C>0$ such that for all $0<\e \leq 1$, $\int_X w_{\e} \, dv \geq -C$ \cite{GZ1}.
It follows from the concavity of the logarithm that
$$
\log \int_X (\omega+\beta)^n \geq \sup_X u_{\e}+\log \int_X (e^{w_{\e}} \, dv) \geq \sup_X u_{\e} -C.
$$
Thus $(\sup_X u_{\e})$ is bounded as claimed.

We now assert that $(u_{\e})$ is decreasing as $\e$ decreases to $0^+$. Indeed assume
that $0<\e' \leq \e$ and fix $\delta>0$. Note that $u_{\e'},u_{\e}$ are both $(\omega+\e \beta)$-plurisubharmonic.
It follows from the (pluripotential) comparison principle that
$$
\int_{(u_{\e'} \geq u_{\e}+\delta)} (\omega+\e \beta +dd^c u_{\e'})^n 
\leq \int_{(u_{\e'} \geq u_{\e}+\delta)} (\omega+\e \beta +dd^c u_{\e})^n .
$$
Since 
$$
(\omega+\e \beta +dd^c u_{\e'})^n \geq (\omega+\e' \beta +dd^c u_{\e'})^n
\geq e^{\delta} (\omega+\e \beta +dd^c u_{\e})^n
$$
on the set $(u_{\e'} \geq u_{\e}+\delta)$, this shows that the latter set has zero
Lebesgue measure. As $\delta>0$ was arbitrary, we infer $u_{\e'} \leq u_{\e}$.

We let $u=\lim_{\e \rightarrow 0} u_{\e}$ denote the decreasing limit of the functions $u_{\e}$.
By construction this is an $\omega$-psh function. It follows from Proposition 1.2, Theorem 2.1 and
Proposition 3.1 in \cite{EGZ1} that $u$ is bounded and (pluripotential) solution of the Monge-Amp\`ere equation
$$
(\omega+dd^c u)^n=e^u \, v.
$$
This shows that $(\dag)$ is satisfied hence we can use Corollary \ref{cor:dagOK} to conclude
that $u$ is actually continuous and that it is a viscosity solution.

\smallskip

It remains to relax the positivity assumption made on $v$. From now on
$\{\omega\}$ is semi-positive and big and $v$ is a probability measure
with semi-positive continuous density.
We can solve 
$$
(\omega+dd^c \f_{\e})^n=e^{\f_{\e}} [v+\e \b^n]
$$
where $\f_{\e}$ are continuous $\omega$-psh functions and $0<\e \leq 1$. Observe that
$$
e^{\sup_X \f_\e} \geq \frac{\int_X \omega^n}{1+\int_X \b^n}
$$
hence $\sup_X \f_{\e}$ is bounded below. 

It follows again from the concavity of the logarithm that
$M_{\e}:=\sup_X \f_{\e}$ is also bounded from above. Indeed set $\p_{\e}:=\f_{\e}-M_{\e}$.
This is a relatively compact family of non-positive $\omega$-psh functions \cite{GZ1}, thus there exists
$C >0$ such that $\int_X \p_{\e} (v+\beta^n) \geq -C$. Now
$$
\log \left( \int e^{\p_{\e}} \frac{v+\e \beta^n}{\int_X v+\e\beta^n} \right)  \geq \int \p_{\e} \frac{v+\e \beta^n}{\int_X v+\e\beta^n} 
\geq \int \p_{\e} (v+\beta^n)
\geq -C
$$
yields 
$$
\log \int_X \omega^n \geq M_{\e}+\log \left[1+\e \int_X \beta^n \right] -C
$$
so that $(M_{\e})$ is uniformly bounded.

We infer that $(\f_{\e})$ is relatively compact in $L^1(X)$. 
It follows from Proposition 2.6 and  Proposition 3.1
in \cite{EGZ1} that $(\f_{\e})$ is actually uniformly bounded,
as $\e$ decreases to zero. 

 Lemma 2.3
in \cite{EGZ1}, together with the uniform bound on $(\f_{\e})$ yields,
for any $0 <\delta << 1$,
\begin{eqnarray*}
Cap_{\omega}(\f_{\e}-\f_{\e'} < -2 \delta)
&\leq& \frac{C}{\delta^n} \int_{(\f_{\e}-\f_{\e'}<-\delta)} (\omega+dd^c \f_{\e})^n \\
&\leq& \frac{C}{\delta^{n+1}} \int_X |\f_{\e}-\f_{\e'}|(\omega+dd^c \f_{\e})^n \\
&\leq& \frac{C'}{\delta^{n+1}} \int_X |\f_{\e}-\f_{\e'}| (v+\b^n) 
\end{eqnarray*}

Using Proposition 2.6 in \cite{EGZ1} again and optimizing the value of $\delta$ yields
the following variant of Proposition 3.3, \cite{EGZ1},
$$
|| \f_{\e}-\f_{\e'} ||_{L^{\infty}} \leq C 
\left( || \f_{\e}-\f_{\e'}||_{L^1} \right)^{\frac{1}{n+2}}.
$$
Thus, if $(\epsilon_n)$ is a sequence decreasing
to zero as $n$ goes to $+\infty$ such that $(\f_{\e_n})_n$ converges in $L^1$,
 $(\f_{\e_n})$ is actually a Cauchy sequence of continuous functions,
 hence it uniformly converges,
 to the unique continuous pluripotential solution $\f$
of $(DMA_v^1)$. From this, it follows that
$(\f_{\e})$ has a unique cluster value 
in $L^1$ when $\epsilon$ decreases to $0$
hence converges in $L^1$. The preceding argument yields uniform convergence. 

Theorem \ref{perron}
insures that the $\f$ is also a viscosity subsolution. Remark
6.3 p. 35 in \cite{CIL} actually enables one to conclude
 that $\f$ is indeed a viscosity solution. 
\end{proof}

\begin{coro}
If $X^{can}$ is a canonical model of a general type projective manifold then the canonical singular KE metric
on $X^{can}$ of \cite{EGZ1} has continuous potentials. 
\end{coro}

\begin{proof}
This is a straightforward consequence of the above theorem, working in  a log resolution of $X^{can}$,
where $\omega=c_1(K_X,h)$ is the pull-back of the Fubini-Study form from $X^{can}$
and $v=v(h)$ has continuous semi-positive density, since $X^{can}$ has canonical
singularities.
\end{proof}

\subsection{Continuous Ricci flat metrics}
 
We now turn to the study of  the degenerate equations $(DMA_v^0)$
$$
(\omega+dd^c \f)^n=v
$$
on a given compact K\"ahler manifold $X$. 
Here $v$ is a continuous volume form
with semipositive density and $\omega$ is a smooth semipositive closed real $(1,1)$ form on $X$.
We assume that $v$ is normalized so that 
$$
v(X)=\int_X \omega^n.
$$
This is an obvious necessary condition in order to solve the equation
$$
(\omega+dd^c \f)^n=v
$$
on $X$. Bounded solutions to such equations have been provided in \cite{EGZ1} when 
$v$ has $L^p$-density, $p>1$, by adapting the arguments of \cite{Kol}.
Our aim here is to show that these are actually {\it continuous}.
We treat here the case of continuous densities, as this is required in
the viscosity context, and refer the reader to section \ref{sec:moreflat}
for more general cases.

\begin{theo} \label{thm:flat}
The pluripotential solutions to $(DMA_v^0)$ are viscosity solutions, hence they are continuous.
\end{theo}

The plan is to combine the viscosity approach for the family of equations
$(\omega+dd^c \f)^n=e^{\e \f} v$, together with the pluripotential tools developed in
\cite{Kol,Ceg2,GZ1,EGZ1,EGZ2}.

\begin{proof}
For $\e>0$ we let $\f_{\e}$ denote the unique viscosity (or equivalently pluripotential) $\omega$-psh 
continuous solution of the equation
$$
(\omega+dd^c \f_{\e})^n=e^{\e \f_{\e}} v.
$$
Set $M_{\e}:=\sup_X \f_{\e}$ and $\p_{\e}:=\f_{\e}-M_{\e}$. The latter form a relatively compact
family of $\omega$-psh functions \cite{GZ1}, hence there exists $C>0$ such that
$$
\int_X \p_{\e} dv =\int_X (\f_{\e}-M_{\e}) dv \geq -C, \text{ for all } \e>0.
$$
Observe that $M_{\e} \geq 0$ since $v(X)=\int_X \omega^n=:V$. The concavity of the logarithm yields
$$
0=\log \left( \int_X e^{\e \f_{\e}} \frac{dv}{V} \right) \geq \frac{1}{V} \int \e \f_{\e} dv 
$$
therefore
$$
0 \geq \int \f_{\e} dv \geq -C +V M_{\e}
$$
i.e. $(M_{\e})$ is uniformly bounded. We infer that $(\f_{\e})$ is relatively compact in $L^1$ and the Monge-Amp\`ere
measures $(\omega+dd^c \f_{\e})^n$ have uniformly bounded densities in $L^{\infty}$.
Once again Proposition 2.6 and (a variant of) Proposition 3.3 in \cite{EGZ1} show that this family of continuous
$\omega$-psh functions is uniformly Cauchy hence converges to a {\it continuous} pluripotential
solution of $(DMA_v^0)$.

This pluripotential solution is also a viscosity solution by \cite{CIL} Remark 6.3.

\smallskip

It is well-known that the solutions of $(DMA_v^0)$ are unique, up to an additive constant.
It is natural to wonder which solution is reached by the the family $\f_{\e}$. Observe that
$\int _X e^{\e \f_{\e}} dv=\int_X dv=\int _X \omega^n$ thus
$$
0=\int_X \frac{e^{\e \f_{\e}}-1}{\e} dv=\int_X \f_{\e} dv+o(1)
$$
hence the limit $\f$ of $\f_{\e}$ as $\e$ decreases to zero is the unique solution of $(DMA_v^0)$
that is normalized by $\int_X \f \, dv=0$.
\end{proof}

Note that the way we have produced solutions (by approximation through the non flat case)
is independent of \cite{Aub,Yau}.

\begin{coro}
Let $X$ be a compact $\Q$-Calabi-Yau K\"ahler space.
Then the Ricci-flat singular metrics constructed in \cite{EGZ1}, Theorem 7.5, have 
continuous potentials.
\end{coro}

\section{Concluding remarks}

\subsection{The continuous Calabi conjecture}

The combination of viscosity methods and pluripotential techniques
yields a soft approach to solving degenerate complex Monge-Amp\`ere
equations of the form
$$
(\omega+dd^c \f)^n=e^{\e \f} v
$$
when $\e \geq 0$.

Recall that here $X$ is a compact K\"ahler n-dimensional manifold, 
$v$ is a semi-positive volume form with continuous density and $\omega$ is 
smooth closed real $(1,1)$-form whose cohomology class is  semi-positive and big
(i.e. $\{\omega\}^n>0$).

Altogether this provides an alternative and independent approach to Yau's solution of the Calabi conjecture \cite{Yau}:
we have only used upper envelope constructions (both in the viscosity and pluripotential sense), 
a global (viscosity) comparison principle  and Kolodziej's pluripotential 
techniques \cite{Kol, EGZ1}.

It applies to degenerate equations but yields solutions that are merely continuous (Yau's work yields
smooth solutions, assuming the cohomology class $\{\omega\}$ is K\"ahler and the measure $v$ is
both positive and smooth).

Note that a third (variational) approach has been studied recently in \cite{BBGZ}. It applies to even
more degenerate situations, providing solutions with less regularity (that belong to the so called
class of finite energy).

\subsection{More continuous solutions} \label{sec:moreflat}

Let $X$ be a compact K\"ahler manifold,
$v=f dV_0$ a non negative measure which is absolutely continuous with respect to some volume 
form $dV_0$ on $X$, and $\omega$ a smooth semi-positive closed real $(1,1)$-form on $X$
with positive volume.
We assume $v$ is normalized so that
$$
v(X)=\int_X \omega^n,
$$
where $n=\dim_{\C} X$.

When $\omega$ is K\"ahler, Kolodziej has shown in \cite{Kol} that there exists a
unique continuous $\omega$-plurisubharmonic function $\f$ such that
$$
(\omega+dd^c \f)^n=v
\text{ and } \int_X \f \, dV_0=0,
$$
as soon as the density $f$ is ``good enough'' (i.e. belongs to some Orlicz class, e.g.
$L^p, p>1$, is good enough).

This result has been extended to the case where $\omega$ is merely semi-positive
in \cite{EGZ1}, but for the continuity statement which now follows from
the viscosity point of view developed in the present 
article: it suffices to approximate the density $f$ by smooth positive densities
$f_{\e}$ (using normalized convolutions) and to show, as in the proof of Theorems
\ref{thm:exp}, \ref{thm:flat} that the corresponding continuous solutions form a Cauchy family
of continuous functions. We leave the details to the reader.

\subsection{The case of a big class}

Our approach applies equally well to a slightly more degenerate situation.
We still assume here that $(X,\omega_X)$ is a compact K\"ahler manifold of dimension $n$, but $v=fdV_0$ is merely assumed
to have density $f \geq 0$ in $L^{\infty}$ and moreover the smooth real closed $(1,1)$-form $\omega$ is no longer
assumed to be semi-positive: we simply assume that its cohomology class $\alpha:=[\omega] \in H^{1,1}(X,\R)$
is {\it big}, i.e. contains a K\"ahler current.

It follows from the work of Demailly \cite{Dem2} that one can find a K\"ahler current in $\alpha$ with analytic 
singularities: there exists an $\omega$-psh function $\p_0$ which is smooth in a Zariski open set
$\Omega_{\a}$ and has logarithmic singularities of analytic type along $X \setminus \Omega_{\a}=\{\p_0=-\infty\}$,
such that $T_0=\omega+dd^c \p_0 \geq \e_0 \omega_X$ dominates the K\"ahler form $\e_0 \omega_X$,
$\e_0>0$.

We refer the reader to \cite{BEGZ} for more preliminary material on this situation. Our aim here is to show that one can
solve $(DMA^{1}_v)$ in a rather elementary way by observing that the (unique) solution is the upper envelope of subsolutions.
We let 
$$
{\mathcal F}:=\left\{ \f \in PSH(X,\omega) \cap L^{\infty}_{loc}(\Omega_{\alpha}) \, / \, (\omega+dd^c \f)^n \geq e^{\f} v
\text{ in } \Omega_{\a} \right\}
$$
denote the set of all (pluripotential) subsolutions to $(DMA^{1}_v)$ (which only makes sense in $\Omega_{\a}$).

Observe that ${\mathcal F}$ is not empty: since $T_0^n$ dominates a volume form and $v$ has density in
$L^{\infty}$, the function $\p_0-C$ belongs to ${\mathcal F}$ for $C$ large enough. We assume for
simplicity $C=0$ (so that $\p_0 \in {\mathcal F}$) and set
$$
{\mathcal F}_0:=\{ \f \in {\mathcal F} \, / \, \f \geq \p_0 \}.
$$

\begin{prop}
The class ${\mathcal F}_0$ is uniformly bounded on $X$.

It is compact (for the $L^1$-topology) and convex.
\end{prop}

\begin{proof}
We first show that ${\mathcal F}_0$ is uniformly bounded from above (by definition it is
bounded from below by $\p_0$). We can assume without loss of generality that $v$ is
normalized so that $v(X)=1$. Fix $\p \in {\mathcal F}_0$. It follows from the convexity
of the exponential that
$$
\exp \left( \int \p dv \right) \leq \int e^{\p} dv \leq \int (\omega+dd^c \p)^n \leq Vol(\a).
$$
All integrals here are computed on the Zariski open set $\Omega_{\a}$.
We refer the reader to \cite{BEGZ} for the definition of the volume of a big class.

We infer
$$
\sup_X \p \leq \int \p dv +C_v \leq \log Vol(\a)+C_v,
$$
where $C_v$ is  a uniform constant that only depends on the fact that all $\omega$-psh functions are integrable
with respect to $v$ (see \cite{GZ1}). This shows that ${\mathcal F}_0$ is uniformly bounded from
above by a constant that only depends on $v$ and $Vol(\a)$.

We now check that ${\mathcal F}_0$ is  compact for the $L^1$-topology.
Fix $\p_j \in {\mathcal F}_0^{\N}$. We can extract a subsequence
that converges in $L^1$ and almost everywhere to a function $\p \in PSH(X,\omega)$.
Since $\p \geq \p_0$, it has a well defined Monge-Amp\`ere measure in $\Omega_{\a}$
and we need to check that $(\omega+dd^c \p)^n \geq e^{\p} v$.

Set $\p_j':=(\sup_{l \geq j} \p_l )^*$. These are functions in ${\mathcal F}_0$ which decrease to $\p$.
It follows from a classical inequality due to Demailly that
$$
(\omega+dd^c \p_j')^n \geq e^{\inf_{l \geq j} \p_l} v
$$
Letting $j \rightarrow +\infty$ shows that $\p \in {\mathcal F}_0$, as claimed.

The convexity of ${\mathcal F}_0$ can be shown along the same lines. We won't need it here
so we let reader check that this easily follows from the inequalities obtained in
\cite{Din}.
\end{proof}

It follows that
$$
\p:=\sup\{ \f \, / \,  \f \in {\mathcal F}_0 \},
$$
the upper envelope of pluripotential subsolutions to $(DMA_v^1)$, is a well defined $\omega$-psh function
which is locally bounded in $\Omega_{\a}$.

\begin{theo}
The function $\p$ is a pluripotential solution to $(DMA_v^1)$.   
\end{theo}

\begin{proof}
In the sequel we shall say (for short) that an $\omega$-psh function $\f$ is bounded iff it is locally bounded
in the Zariski open set $\Omega_{\a}$.

  By Choquet's lemma, we can
find a sequence $\p_j \in {\mathcal F}_0$ of bounded $\omega$-psh (pluripotential) subsolutions
such that 
$$
\p=\left( \lim_{j \rightarrow +\infty} \p_j \right)^*.
$$

Observe that the family of bounded pluripotential subsolutions is stable under taking maximum:
assume $w_1,w_2$ are two such subsolutions and set $W_c:=\max(w_1+c,w_2)$, then
the (pluripotential and local) comparison principle yields
\begin{eqnarray*}
(\omega+dd^c W_c)^n &\geq& {\bf 1}_{\{ w_1+c>w_2\}} (\omega+dd^c w_1)^n
+{\bf 1}_{\{ w_1+c<w_2\}} (\omega+dd^c w_2)^n \\
&\geq&{\bf 1}_{\{ w_1+c \neq w_2\}} e^{W_c} v
\end{eqnarray*}
Now for all but countably many $c$'s, the sets $(w_1+c=w_2)$ have zero
$v$-measure, thus by continuity of the Monge-Amp\`ere operator under
decreasing sequences, we infer
$$
(\omega+dd^c \max[w_1,w_2])^n \geq e^{\max[w_1,w_2]} v.
$$

This shows that we can assume the $\p_j$'s form an increasing sequence
of subsolutions. Finally we use a local balayage procedure to show that
$\p$ is indeed a pluripotential solution to $(DMA_v^1)$. Fix $B$ an 
arbitrary small ``ball'' in $X$ (image of a euclidean ball under a local
biholomorphism) and let $\p_j'$ denote the solution of the local Dirichlet
problem
$$
(\omega+dd^c \p_j')^n=e^{\p_j'} v \text{ in } B
\; \;
\text{ and }
\; \;
\p_j' \equiv \p_j \text{ on } \partial B.
$$
We extend $\p_j'$ to $X$ by setting $\p_j' \equiv \p_j$ in $X \setminus B$.

That such a problem indeed has a solution follows from an adaptation
of the corresponding ``flat'' Dirichlet problem of Bedford and Taylor
\cite{BT1,BT2}, as was considered by Cegrell \cite{Ceg1}.

Note that the $\p_j'$ are still subsolutions. It follows from the (pluripotential) comparison
principle that $\p_j' \geq \p_j$ and $\p_{j+1}' \geq \p_j'$.
Thus the increasing limit of the $\p_j'$s equals again $\p$.
Since the Monge-Amp\`ere operator is continuous under increasing sequences \cite{BT2},
this shows that $\p$ is a pluripotential solution of $(DMA_v^1)$ in $B$, hence
in all of $X$, as $B$ was arbitrary.
\end{proof}

\begin{rem}
The situation considered above covers in particular the construction
of a K\"ahler-Einstein current on  a variety $V$ with
ample canonical bundle $K_V$ and canonical singularities, since the
canonical volume form becomes, after passing to a desingularisation $X$, a volume form
$v=f dV_0$ with density $f \in L^{\infty}$.

The more general case of log-terminal singularities yields density $f \in L^p$, $p>1$.
One can treat this case by an easy approximation argument:
setting $f_j=\min (f,j) \in L^{\infty}$, one first solves 
$(\omega+dd^c \f_j)^n=e^{\f_j} f_j dV_0$ and observe (by using the comparison principle)
that the $\f_j's$ form a decreasing sequence which converges to the unique
solution of $(\omega+dd^c \f)^n=e^{\f} f dV_0$.

Once again the problem $(DMA_v^0)$ can be reached by first solving 
$(DMA_v^{\e})$, $\e>0$, and then letting $\e$ decrease to zero.
\end{rem}

\subsection{More comparison principles}

Let again $B \subset \C^n$ denote the open unit ball
 and let $B'=(1+\eta)B$ with $\eta>0$
be a slightly larger open ball. Let $u, u'\in PSH(B')$ be  
plurisubharmonic functions. By convolution with an adequate non negative kernel
of the form
$\rho_{\epsilon}(z)=\epsilon^{-2n}\rho_1(\frac{z}{\epsilon}) $
we construct $(u_{\epsilon})_{\eta>\epsilon>0}$ a family of smooth
plurisubharmonic functions decreasing to $u$ as $\epsilon$ decreases to $0$.

\begin{lem}\label{mol}
$$ \forall z \in B \  u (z) + u'(z)=
\lim_{n\to \infty} \sup\{ u'(x)+u_{1/j}(x) | j\ge n, \  |x-z|\le 1/n \}
$$
\end{lem}

\begin{proof}
Indeed, we have, if $2/n <\eta$: 
 \begin{eqnarray*}
u(z)+u'(z)&\le& \sup \{ u'(z)+ u_{1/j}(z) | j\ge n \}
\\
&\le &
\sup\{ u'(x)+u_{1/j}(x) | j\ge n, \  |x-z|\le 1/n \} \\  
&\le &\sup\{u'(x)+ u(x)| \quad |x-z| \le 2/n \}.
 \end{eqnarray*}
Since $u+u'$ is upper semicontinuous, we have: 
$$ u (z) + u'(z)=(u+u')^*(z)= \lim_{n\to \infty} \sup\{ u+u'(x)| \quad |x-z| \le 2/n \}.$$
\end{proof}

\begin{lem} \label{lemcle}
Let $\f$ a bounded psh function on $B$ and $v$ 
a continuous non negative volume form 
such that $e^{-\f}(dd^c\f)^n \ge v$ in the viscosity sense. 

Let $\psi$ be a bounded psh function and $w$ 
a continuous positive volume form, both defined on $B'$
 such that $(dd^c\psi)^n \ge w$. 

Then
$\exists C,c >0$ depending only on 
$\| \psi \| _{L^{\infty}}, \| \f \|_{L^{\infty}}$ 
such that for every $\epsilon \in [0,1]$ 
$\Phi=\f+\epsilon \psi$ satisfies: 

$$ e^{-\Phi} (dd^c\Phi)^n \ge (1-\epsilon)^n e^{-C\epsilon}  v + c \epsilon^n w 
$$
in the viscosity sense in $ B$. 
\end{lem}

\begin{proof} We may assume $\epsilon >0$ and $w$ to be smooth.
Let us begin by the case when $\psi$ is of class $C^2$. 
Let $x_0\in B$ and $q\in C^2$ such that $q(x_0)=\Phi (x_0)$ 
and $\Phi-q$ has a local maximum at $x_0$. 
Then, $\f- (q-\epsilon \psi)$ has a local maximum at $x_0$.

  We deduce:
$$ dd^c (q-\epsilon \psi)_{x_0} \ge 0 $$
$$ e^{- q(x_0) + \epsilon \psi(x_0) } (dd^c (q-\epsilon \psi))_{x_0})^n \ge v_{x_0}.$$
Using the inequality $(dd^c q)^n_{x_0} \ge (dd^c (q-\epsilon \psi))_{x_0})^n + \epsilon^n (dd^c\psi)^n$, we conclude. 

\smallskip

We now treat the general case.
Since $\psi$ is defined on $B'$ we can construct by the above classical 
mollification  a sequence of $C^2$ psh functions $(\psi_{1/k})$ converging to $\psi$
as $k$ goes to $+\infty$. 

We know from the proof of Proposition \ref{pro:visc=pluripot} 
that $(dd^c\psi_k)^n \ge ((w^{1/n} )_{1/k} )^n=w_k$ in both the pluripotential 
and viscosity sense.

We conclude from the previous case  that
 $\Phi_k= \f + \epsilon \psi_k$ satisfies
$$
c \epsilon^n w_k + (1-\epsilon)^n e^{-C\epsilon}  v \le e^{-\Phi_k} (dd^c \Phi_k)^n
$$
in the viscosity sense. 
Since $w_k>0$, there is no difference between subsolution
of $DMA$ and of $DMA_+$ hence, we have:
 $$ c \epsilon^n w_k + (1-\epsilon)^n e^{-C\epsilon} 
 v -e^{-\Phi_k} (dd^c \Phi_k)_+ ^n \le 0$$ in the viscosity sense. 

By Lemma 6.1 p. 34 and  remark 6.3 p. 35 in \cite{CIL}, we conclude that 
$$\bar\Phi=
\limsup_{n\to \infty} \sup\{ \Phi_j(x) | j\ge n, \  |x-z|\le 1/n \} $$
 satisfies the limit inequation
$$ e^{-\bar\Phi} (dd^c\bar\Phi)_+^n
 \ge (1-\epsilon)^n e^{-C\epsilon}  v + c \epsilon^n w 
$$ in the viscosity sense\footnote{Here we use the fact that 
$DMA_+$ is a continuous equation}.  
Now Lemma \ref{mol} implies that $\bar \Phi=\Phi$. 
Since $w>0$, the proof is complete. 
\end{proof}

\begin{theo}
Let $X$ be a compact K\"ahler manifold 
and $\omega \ge 0$ be a semi-k\"ahler smooth form. 

Then, the global viscosity comparison principle holds 
for $(DMA^1_v)$ for any non negative continuous probability measure $v$. 
\end{theo}

\begin{proof} This is a variant of
 the  argument sketched in \cite{IL} sect. V.3 p. 56. 

Let $\overline{u}$  be a  supersolution 
and  $\underline{u}$ be a 
subsolution. 
Perturb  the supersolution $\overline{u}$ 
setting  $\overline{u}_{\delta}=\overline{u}+\delta$. 
This  $\overline{u}_{\delta}$ is a supersolution 
to $(DMA^1_{\tilde w})$ for every continuous volume form $\tilde w$
such that $\tilde w\ge e^{-\delta}v$. 

Choose $w>0$ a continuous positive probability measure. 
Assuming w.l.o.g. that $\int_X\omega^n=1$, we can construct $\psi$ a continuous quasiplurisubharmonic functions 
such that, in the vicosity sense
$$
(\omega+dd^c\psi)^n=w.
$$

 Perturb  the 
subsolution $\underline{u}$ setting
 $$
 \underline{u}_{\epsilon}=(1-\epsilon)\underline{u} +\epsilon \psi.
 $$ 
By Lemma \ref{lemcle}, $\underline{u}_{\epsilon}$ satisfies, in the viscosity sense
$$
e^{-(1+\epsilon)u} (\omega+dd^c u)^n \ge 
\left(\frac{1-\epsilon}{1+\epsilon}\right)^ne^{-C\epsilon}v+ c\left(\frac{\epsilon}{1+\epsilon}\right)^nw 
$$
 This in turn implies that $\underline{u}_{\epsilon}$ satisfies, in the viscosity sense: 
$$
e^{-u} (\omega+dd^c u)^n \ge e^{-\epsilon \|u \|_{\infty}}  
\left[ \left(\frac{1-\epsilon}{1+\epsilon}\right)^ne^{-C\epsilon}v+ c\left(\frac{\epsilon}{1+\epsilon}\right)^nw \right].
$$

 Hence $\underline{u}_{\epsilon}$ satisifies, in the viscosity sense: 
$$
e^{-u} (\omega+dd^c u)^n \ge \tilde w$$ whenever 
$\tilde w \le e^{-\epsilon \|u \|_{\infty}}(
(\frac{1-\epsilon}{1+\epsilon})^ne^{-C\epsilon}v+ c(\frac{\epsilon}{1+\epsilon})^nw).
$

Choosing
 $1\gg \delta \gg \epsilon >0$, we find a continuous volume form 
$\tilde w >0$ such
that   $\overline{u}_{\delta}$
is a supersolution and $\underline{u}_{\epsilon}$ is a viscosity subsolution 
of $e^{-u}(\omega+dd^c u)^n=\tilde w$. Using the viscosity  
comparison principle for $\tilde w$, we conclude that
 $\overline{u}_{\delta}\ge \underline{u}_\epsilon$. Letting 
$\delta\to 0$, we infer $\overline{u}\ge \underline{u}$.
\end{proof}

This comparison principle has been inserted here for completeness. 
It could have been used instead of the pluripotential-theoretic arguments to establish existence 
of a viscosity solution in the case $v\ge 0$ of Theorem  \ref{thm:exp}. 
This could be useful in dealing with similar problems where pluripotential tools are
less efficient.

\subsection{Viscosity supersolutions of $(dd^c \f)^n=v$}

Assume $\Omega$ is an euclidean ball.
%(or a bounded pseudoconvex domain in $\C^n$).
 Given $\f$ a bounded function, its plurisubharmonic projection
$$
P(\f)(x)= P_{\Omega}(\f)(x) := \left(\sup \{ \p(x) \, / \, \p \text{ psh on $\Omega$ and } \p \leq \f \}\right)^*,
$$
 is the greatest psh function that lies below $\f$ on $\Omega$. 
If $\f$ is upper semi-continuous on $\Omega$ there is no need of
 upper regularization and the upper enveloppe is $\leq \f$ on $\Omega$.

\begin{lem} \label{pro:super} 
\text{ }

 1) Let $\psi$ be a bounded plurisubharmonic function satisfying $(dd^c \psi)_{BT}^n \leq v$ 
 on $\Omega$. Then its lower semi-continuous regularization $\psi_*$ 
is a viscosity supersolution of the equation  $(dd^c  
 \f)^n=v$ on $\Omega$. 
 
 2) Let $\f$ be a continuous viscosity supersolution 
of the equation  $(dd^c \f)^n=v$ on $\Omega$. 
Then $ \psi := P (\f)$ is  a continuous 
 plurisubharmonic viscosity supersolution of the equation $(dd^c \psi)^n = v$ on $\Omega$. 
 
 3) Let $\f$ be a $C^2$-smooth viscosity supersolution of $(dd^c \f)^n=v$ in $\Omega$.
 Its plurisubharmonic projection
$P(\f)$ satisfies $(dd^c P(\f))^n_{BT}\le v$. 

 \end{lem}

\begin{proof}
1. We use the same idea as in the proof of Proposition 1.5. Assume $\psi \in PSH\cap L^{\infty} (\Omega)$ 
satisfies $(dd^c\psi)^n_{BT} \le v$ in the pluripotential sense on $\Omega$. Consider $q$ a ${\mathcal C}^2$ function such that $\psi_*(x_0)=q(x_0)$ and
$\psi_* - q$ achieves a local minimum at $x_0$. We want to prove that
$(dd^cq (x_0))_+^n  \le v (x_0)$.   Assume that $(dd^c q (x_0))_+^n > v_{x_0}$. Then $dd^c q (x_0) \geq 0$ and $(dd^c q (x_0))^n > v_{x_0} > 0$ which implies that $dd^c q (x_0) > 0$.  
Let $q^{\e}:=q - 2 \e (\| x-x_0 \|^2- r^2) - \e r^2$.
 Since $v$ has continuous density, we can choose $\e > 0$ small enough and a small ball $B (x_0,r)$ containing $x_0$ of radius $r>0$ such that  $dd^c q^{\e} > 0$ in $B (x_0,r)$ and $(dd^c q^{\e})^n > v$ on the ball $B (x_0,r)$. Thus we have
$q^{\e} = q - \e r^2 <  \psi_* \leq \psi $ near $\partial B (x_0,r)$ 
while $(dd^c q^{\e})^n_{BT} \ge v \geq (dd^c\psi)^n_{BT}$  on $B (x_0,r)$. 
The comparison principle (Lemma \ref{cpd})
yields $q^{\e} \le \psi$ on $B (x_0,r)$ hence $q^\e (x_0) = \liminf_{x \to x_0} q^\e (x) \leq \liminf_{x \to x_0} \psi (x) = \psi_* (x_0)$ i.e. $q
 (x_0) + \e r^2 \leq \psi_* (x_0) = q (x_0),$ which is a contradiction.   
 
 \smallskip

2. Set $\psi : = P (\f)$, fix a point $x_0 \in \Omega$ and consider a super test function $q$ for $\psi$ at $x_0$ i.e. $q$ is a $C^2$ function on a small ball $B (x_0,r) \subset \Omega$ such that  $\psi (x_0) = q (x_0)$ and $\psi - q$ attains its  minimum  at $x_0$.
We want to prove that $(dd^c q (x_0))_+^n \leq v (x_0)$. Since $\psi \leq \f$, there are two cases:\\
- if $\psi (x_0) = \f (x_0)$ then $q$ is also a super test function for $\f$ at $x_0$ and then  $(dd^c q (x_0))^n_+ \leq v (x_0)$ since $\f$ is a supersolution of the same equation, \\
- if $\psi (x_0) < \f (x_0)$, by continuity of $\f$ there exists a ball $B (x_0,s)$ $0 < s < r$ such that $\psi = P (\f) < \f$ on the ball $B (x_0,s)$ and then
$ (dd^c \psi)^n = 0$ on $B (x_0,s)$  since $(dd^c P(\f))^n$ is supported on the set $\{P (\f) = \f\}$. Therefore $\psi$ is a continuous psh function satisfying the inequality  $(dd^c \psi)^n = 0 \leq v $ in the sense of pluripotential theory on the ball $B (x_0,s)$. 
Assume that $(dd^c q (x_0))_+^n > v (x_0)$. 
Then by definition, $dd^c q (x_0) > 0$ and $(dd^c q (x_0))^n > v (x_0)$. 
Taking $s > 0$ small enough and $\e > 0$ small enough we can assume that
 $q^\e := q - \e (\vert x - x_0\vert^2 - s^2) $ is psh on
 $B (x_0,s)$ and $(dd^c q^\e)^n > v \geq (dd^c \psi)^n$ on
 the ball $B (x_0,s)$ while $q^\e = q \leq \psi$ on $\partial B (x_0,s)$. By the pluripotential
 comparison principle for the complex Monge-Amp\`ere operator, 
it follows that $q^\e \leq \psi$ on $B (x_0,s)$, thus $q (x_0) + \e s^2 \leq \psi (x_0)$, 
which is a contradiction.

\smallskip

3. It is classical (see \cite{BT1}, \cite{Berm}, \cite{Dem1}) 
that under these hypotheses, $P(\f)$ is
a ${C}^{1,1}$-smooth function,
its Monge-Amp\`ere measure $(dd^c P(\f))_{BT}^n$ 
 is concentrated on the set where $P(\f)=\f$, and satisfies
$$
(dd^c P(\f) )_{BT}^n={\bf 1}_{\{P(\f)=\f \}} (dd^c \f )^n.
$$ 
The conclusion follows from the definition of viscosity supersolutions. 
\end{proof}

We do not know whether part 3 of the lemma is valid for less regular supersolutions.

\end{document}